\documentclass[11pt,reqno,a4paper]{amsart}
\usepackage{mathrsfs}

\usepackage{amssymb, bbm}
\usepackage{hyperref}\hypersetup{colorlinks=true}
\usepackage[dvips]{graphicx}
\usepackage{psfrag}

\newtheorem{theorem}{Theorem}[section]
\newtheorem{lemma}[theorem]{Lemma}
\newtheorem{propos}[theorem]{Proposition}
\newtheorem{corol}[theorem]{Corollary}

\theoremstyle{remark}

\newtheorem{definition}[theorem]{Definition}

\numberwithin{equation}{section}

\newcommand{\R}{\mathbb{R}}

\newcommand{\Haus}{{\mathscr H}}
\newcommand{\Leb}{{\mathscr L}}

\newcommand{\res}{\mathop{\hbox{\vrule height 7pt width .5pt depth 0pt
\vrule height .5pt width 6pt depth 0pt}}\nolimits}
\newcommand{\trace}{{\rm Tr}}

\newcommand{\eps}{\varepsilon}
\renewcommand{\div}{{\rm div}\,}
\newcommand{\divy}{{\rm div}_y}

\newcommand{\divtx}{{\rm Div}\,}
\newcommand{\bvloc}{BV_{\mathrm{loc}}}
\newcommand{\weaks}{\stackrel{*}{\rightharpoonup}}
\newcommand{\loc}{\mathrm{loc}}
\newcommand{\tr}{\mathrm{Tr} \,}

\begin{document}

\title[IBVP for the continuity equation]{Initial-boundary value problems for continuity equations with $BV$ coefficients}

\author[G.~Crippa]{Gianluca Crippa}
\address{G.C. Departement Mathematik und Informatik,
Universit\"at Basel, Rheinsprung 21, CH-4051 Basel, Switzerland.}
\email{gianluca.crippa@unibas.ch}
\author[C.~Donadello]{Carlotta Donadello}
\address{C.D. Laboratoire de Math\'ematiques,
Universit\'e de Franche-Comt\'e, 16 route de Gray, F-25030 Besan\c con Cedex,
France.}
\email{Carlotta.Donadello@univ-fcomte.fr}
\author[L.~V.~Spinolo]{Laura V.~Spinolo}
\address{L.V.S. IMATI-CNR, via Ferrata 1, I-27100 Pavia, Italy.}
\email{spinolo@imati.cnr.it}
\maketitle
{
\rightskip .85 cm
\leftskip .85 cm
\parindent 0 pt
\begin{footnotesize}

{\sc Abstract.}
We establish well-posedness of initial-boundary value problems for continuity equations with $BV$ (bounded total variation) coefficients. We do not prescribe any condition on the orientation of the coefficients at the boundary of the domain. We also discuss some examples showing that, regardless the orientation of the coefficients at the boundary, uniqueness may be violated as soon as the $BV$ regularity deteriorates at the boundary. 

%\emph{Draft, \today}



\end{footnotesize}

}

\vspace{.3 cm}

\section{Introduction}
This work is devoted to the study of the initial-boundary value problem for the continuity equation
\begin{equation}
\label{e:cone}
    \partial_t u + \div (bu) = 0, \quad \text{where $b : \, ]0, T[ \times \Omega \to \R^d$ and $u : \,  ]0, T[ \times \Omega \to \R$.}
\end{equation}
In the previous expression, $\Omega \subseteq \R^d$ is an open set, $T>0$ is a real number and $\mathrm{div}$ denotes the divergence computed with respect to the space variable only. 

The analysis of~\eqref{e:cone} in the case when $b$ has low regularity has recently drawn considerable attention: for an overview of some of the main contributions, we refer to the lecture notes by Ambrosio and Crippa~\cite{AmbrosioCrippa}. Here, we only quote the two main breakthroughs due to DiPerna and Lions~\cite{diPernaLions} and to Ambrosio~\cite{Ambrosio:trabv}, which deal with the case when $\div b$ is bounded and $b$ enjoys Sobolev and $BV$ (bounded total variation) regularity, respectively.  More precisely, in~\cite{diPernaLions} and~\cite{Ambrosio:trabv} the authors establish existence and uniqueness results for the Cauchy problem posed by coupling~\eqref{e:cone} with an initial datum in the case when $\Omega= \R^d$. 

In the classical framework where $b$ and $u$ are both smooth up to the boundary, the initial-boundary value problem is posed by prescribing
\begin{equation}
\label{e:clpb}
\left\{ 
\begin{array}{lll}
          \partial_t u + \div (bu) = 0 & \text{in $]0, T [ \times \Omega $} \\ 
         % \phantom{ciao} \\
           u = \bar g & \text{on $\Gamma^-$}\\
            %  \phantom{ciao} \\
         u = \bar u  & \text{at $t =0$},    \\
\end{array}
\right.
\end{equation}
where $\bar u$ and $\bar g$ are bounded smooth functions  
and $\Gamma^-$ is the portion of $]0, T[ \times \partial \Omega$
where the characteristics are entering the domain $]0, T[ \times \Omega$. Note, however, that if $b$ and $u$ are not sufficiently regular (if, for example, $u$ is only an $L^{\infty}$ function), then their values on negligible sets are not, a priori, well defined. In~\S~\ref{ss:df} we provide the distributional formulation of~\eqref{e:clpb} by relying on the theory of normal traces for weakly differentiable vector fields, see the works by Anzellotti~\cite{Anzellotti} and, more recently, by Chen and Frid~\cite{ChenFrid}, Chen, Torres and Ziemer~\cite{ChenTorresZiemer} and by Ambrosio, Crippa and Maniglia~\cite{AmbrosioCrippaManiglia}. 

Our main positive result reads as follows:
\begin{theorem}
\label{t:wp}
Let $\Omega \subseteq \R^d$ be an open set with uniformly Lipschitz  boundary.  Assume that the vector field $b$ satisfies the following hypotheses:
\begin{itemize}
\item[1.] $b \in L^{\infty} (]0, T [ \times \Omega ; \R^d)$;
\item[2.] $\div b \in L^{\infty} (]0, T [ \times \Omega)$;
\item[3.] for every open and bounded set $\Omega_\ast \subseteq  \Omega$, $b \in L^1_\loc([0, T [  ; BV  (\Omega_\ast; \R^d))$. 
\end{itemize}
Then, given $\bar u \in L^{\infty} ( \Omega )$ and $\bar g \in L^{\infty} ( \Gamma^-) $, problem~\eqref{e:clpb} admits a unique distributional solution $u \in L^{\infty} (]0, T [ \times \Omega)$. 
\end{theorem}
Some remarks are here in order. First, we recall that $\Gamma^-$ is a subset of $]0, T[ \times \partial \Omega$ and we point out that $L^{\infty} (\Gamma^-)$ denotes the space 
$L^{\infty} (\Gamma^-,~\Leb^1~\otimes~\Haus^{d-1})$. 

Second, we refer to the book by Leoni~\cite[Definition 12.10]{Leoni} for the definition of open set with uniformly Lipschitz boundary. In the case when $\Omega$ is bounded, the definition reduces to the classical condition that $\Omega$ has Lipschitz boundary. This regularity assumption guarantees that classical results on the traces of Sobolev and $BV$ functions apply to the set $\Omega$, see again Leoni~\cite{Leoni} for an extended discussion. 

Third, several works are devoted to the analysis of the initial-boundary value problem~\eqref{e:clpb}. In particular, we refer to Bardos~\cite{Bardos} for an extended discussion on the case when $b$ enjoys Lipschitz regularity, and to Mischler~\cite{Mischler} for the case when the continuity equation in~\eqref{e:clpb} is the Vlasov equation. Also, we quote reference~\cite{Boyer}, where Boyer establishes uniqueness and existence results for~\eqref{e:clpb} and investigates space continuity properties of the trace of the solution on suitable surfaces. The main assumption in~\cite{Boyer} is that $b$ has Sobolev regularity, and besides this there are the technical assumptions that $\div b \equiv 0$ and that $\Omega$ is bounded. See also the analysis by Girault and Ridgway Scott~\cite{GiraultRidgway} for the case when $b$ enjoys Sobolev regularity and is tangent to the boundary. Note that the extension of Boyer's proof to the case when $b$ has $BV$ regularity is not straightforward. 

Our approach is quite different from Boyer's: indeed, the analysis in~\cite{Boyer} is based on careful estimates on the behavior of $b$ and $u$ close to the boundary and involves the introduction of a system of normal and tangent coordinates at $\partial \Omega$, and the use of a local regularization of the equation. Conversely, as mentioned above, in the present work we rely on the theory of normal traces for weakly differentiable vector fields.  From the point of view of the results we obtain, the main novelties of the present work can be summarized as follows. 
\begin{itemize}
\item We provide a distributional formulation of problem~\eqref{e:clpb} under the solely assumptions that $b \in L^{\infty} (]0, T[ \times \Omega; \R^d)$ and $\div b$ is a locally finite Radon measure, see Lemma~\ref{l:wt} and Definition~\ref{d:rf} in \S~\ref{ss:df}. Conversely, in~\cite{Boyer} the distributional formulation of~\eqref{e:clpb} requires that $b$ enjoys Sobolev regularity. 
\item We establish well-posedness of~\eqref{e:clpb} (see Theorem~\ref{t:wp}) under the assumptions that $b$ enjoys $BV$ regularity, while in~\cite{Boyer} Sobolev regularity is required. Note, however, that the main novelty in Theorem~\ref{t:wp} is the uniqueness part, since existence can be established under the solely hypotheses that $b \in L^{\infty} (]0, T[ \times \Omega; \R^d)$ and $\div b \in L^{\infty}(]0, T[ \times \Omega)$ by closely following the same argument as in~\cite{Boyer}, see~\cite{CDS:hyp} for the technical details. We point out in passing that, for the Cauchy problem, the extension of the uniqueness result from Sobolev to $BV$ regularity is one of the main achievement in Ambrosio's paper~\cite{Ambrosio:trabv}. Also, this extension is crucial in view of the applications to some classes of nonlinear PDEs like systems of conservation laws in several space dimensions, see the lecture notes by De Lellis~\cite{CDL:notes} and the references therein.   
\item We exhibit some counterexamples (see Theorems~\ref{t:cex} and~\ref{t:cex2} and Corollary~\ref{c:dpt} below) showing that, regardless the orientation of $b$ at the boundary, uniqueness may be violated as soon as $b$ enjoys $BV$ regularity in every open set $\Omega_\ast$ compactly contained in $\Omega$, but the regularity deteriorates at the boundary of $\Omega$. Also, as the proof of Theorem~\ref{t:cex2} shows,  if $BV$ regularity deteriorates at the domain boundary, it may happen that the normal trace of $b$ at $\partial \Omega$ is identically zero, while the normal trace of $bu$ is identically $1$, see \S~\ref{ss:df} for the definition of normal trace of $b$ and $bu$.  
\item In~\cite[\S~7.1]{Boyer}, Boyer establishes a space continuity property for the solution of~\eqref{e:clpb} in directions trasversal to the vector field $b$ under the assumption that $b$ enjoys Sobolev regularity. Proposition~\ref{p:sc} in the present work ensures that an analogous property holds under $BV$ regularity assumptions. The property we establish is loosely speaking the following: assume $\Sigma_r$ is a family of surfaces which continuously depend on the parameter $r$ and assume moreover that the surfaces are all transversal to a given direction.  Then the normal trace of the vector field $u b$ on $\Sigma_r$ strongly converges to the normal trace of $ub$ on $\Sigma_{r_0}$ as $r \to r_0$.
\end{itemize}
Here is our first counterexample. In the statement of Theorem~\ref{t:cex}, $\tr b$ denotes the normal trace of $b$ along the outward pointing, unit normal vector to $\partial \Omega$, as defined in~\S~\ref{ss:df}. 
\begin{theorem}
\label{t:cex}
      Let $\Omega$ be the set $\Omega: = ]0, + \infty[ \times \R^2$. Then there is a vector field ${b: ]0, 1[ \times \Omega \to \mathbb R^3}$ such that
      \begin{itemize}
       \item[i)] $ b \in L^{\infty} (]0,1[ \times \Omega; \mathbb R^3)$;
       \item[ii)] $\div b \equiv 0$;
       \item[iii)] for every open and bounded set $\Omega_\ast$ such that its closure 
       $\bar \Omega_\ast \subseteq \Omega$, we have $b \in  L^1([0, 1 [  ; BV  (\Omega_\ast; \R^3))$;
        \item[iv)] $\tr b \equiv - 1$ on $]0, 1[ \times \partial \Omega$;
       \item[v)]  the initial-boundary value problem
      \begin{equation}
       \label{e:cex:ibvp}
       \left\{
        \begin{array}{lll}
              \partial_t u + \div (bu ) = 0 & \text{in  $]0, 1[ \times \Omega$} \\
              u = 0 & \text{on $ ]0, 1[ \times \partial \Omega$} \\
              %\phantom{\displaystyle{\int}}\\
               u = 0 & \text{at $ t=0 $}   
               \phantom{\int} \\
        \end{array}
       \right.  
      \end{equation}
      admits infinitely many different solutions. 
\end{itemize}
\end{theorem}
Some remarks are here in order. First, since the vector field $b$ is divergence-free, then any solution of~\eqref{e:cex:ibvp} is a solution of the transport equation $$\partial_t u + b \cdot \nabla u=0$$ satisfying 
zero boundary and initial conditions. Second, the proof of Theorem~\ref{t:cex} uses an intriguing construction due to Depauw~\cite{Depauw}.  

Finally, note that property iv) in the statement of Theorem~\ref{t:cex} states that the vector field $b$  is inward pointing at the boundary $\partial \Omega$. This fact is actually crucial for our argument because it allows us to build on Depauw's construction.

When the vector field is outward pointing, one could heuristically expect that the solution would not be affected by the loss of regularity of $b$ at the domain boundary. Indeed, in  the smooth case the solution is simply ``carried out" of the domain along the characteristics and, consequently, the behavior of the solution inside the domain is not substantially affected by what happens close to the boundary. Hence, one would be tempted to guess that, even in the non smooth case, when $\tr b >0$ on the boundary the solution inside the domain is not affected by boundary behaviors and uniqueness should hold even when the $BV$ regularity of $b$ deteriorates at the boundary. The example discussed in the statement of Theorem~\ref{t:cex2} shows that this is actually not the case and that, even if $b$ is outward pointing at $\partial \Omega$, then uniqueness may be violated as soon as the $BV$ regularity deteriorates at the boundary.
\begin{theorem}
\label{t:cex2}
      Let $\Omega$ be the set $\Omega: = ]0, + \infty[ \times \R^2$. Then there is a vector field ${b: ]0, 1[ \times \Omega \to \mathbb R^3}$ such that
      \begin{itemize}
       \item[i)] $ b \in L^{\infty} (]0,1[ \times \Omega; \mathbb R^3)$;
       \item[ii)] $\div b \equiv 0$;
       \item[iii)] for every open and bounded set $\Omega_\ast$ such that its closure 
       $\bar \Omega_\ast \subseteq \Omega$, we have $b \in  L^1([0, 1 [  ; BV  (\Omega_\ast; \R^3))$;
       \item[iv)] $\tr b \equiv 1$ on $]0, 1[ \times \partial \Omega$;       
       \item[v)] the initial-boundary value problem
      \begin{equation}
       \label{e:cex:ibvp2}
       \left\{
        \begin{array}{lll}
              \partial_t u + \div (bu ) = 0 & \text{in  $]0, 1[ \times \Omega$} \\
                             u = 0 & \text{at $ t=0$}   \phantom{\int} \\
        \end{array}
       \right.  
      \end{equation}
      admits infinitely many different solutions. 
\end{itemize}
\end{theorem}
We make some observations. First, by a trivial modification of the proof one can exhibit a vector field $b$ satisfying properties i), ii), iii) and v) above and, instead of property iv), $\tr b \equiv 0$ on $]0, 1[ \times \partial \Omega$. Hence, even in the case when $b$ is tangent at the domain boundary, uniqueness may be violated as soon as the $BV$ regularity deteriorates at the domain boundary.

Second, the proof of Theorem~\ref{t:cex2} does not use Depauw's example~\cite{Depauw}. The key point is constructing a non trivial solution of~\eqref{e:cex:ibvp2} such that $u(t, x) \ge 0$ for a.e. $(t, x) \in ]0, 1[ \times \Omega$ and $\tr \, (bu) < 0$ (note that $\tr b > 0$ by property iv) in the statement of the theorem). Heuristically, such solution ``enters'' the domain $\Omega$, although the characteristics are outward pointing at the boundary.

Third, since $\tr b \equiv 1$ on $]0, 1 [ \times \partial \Omega$, then in the formulation of the initial-boundary value problem~\eqref{e:cex:ibvp2} we do not prescribe the value of the solution $u$ at the boundary. In the proof of Theorem~\ref{t:cex2}  we exhibit infinitely many different solutions of~\eqref{e:cex:ibvp2} and in general different solutions attain different values on $]0, 1[ \times \partial \Omega$. However, by refining the proof of Theorem~\ref{t:cex2} we obtain the following result.
\begin{corol}
\label{c:dpt}
        Let $\Omega$ be the set $\Omega : = ]0, + \infty [ \times \R^2$, then there is a vector field $b~:~]0, 1[ \times \Omega \to \R^3$ satisfying requirements $\mathrm{i)}, \dots, \mathrm{v)}$ in the statement of Theorem~\ref{t:cex2} and such that the initial-boundary value problem~\eqref{e:cex:ibvp2} admits infinitely many solutions that satisfy $\tr (bu) \equiv 0$ on $]0, 1[ \times \partial \Omega$.  
\end{corol}
The additional condition 
$\tr (bu) \equiv 0$ in the corollary can be heuristically interpreted as (a weak version of)
$u \equiv 0$ on $]0, 1[ \times \partial \Omega$. 

We also point out that, again by a trivial modification of the proof, one can exhibit a vector field $b$ satisfying properties i), ii), iii) and v) in the statement of Corollary~\ref{c:dpt} and, instead of property iv), $\tr b \equiv 0$ on $]0, 1[ \times \partial \Omega$. Also, for any given real constant $c$, one can actually construct infinitely many solutions of~\eqref{e:cex:ibvp2} that satisfy $\tr (bu) = c$ on $]0, 1[ \times \partial \Omega$.

\subsection*{Outline}
The paper is organized as follows. In \S~\ref{s:pre} we recall some results on normal traces of vector fields established in~\cite{AmbrosioCrippaManiglia}. In \S~\ref{s:pwp} we establish the uniqueness part of the proof of Theorem~\ref{t:wp} and the space continuity property. In \S~\ref{s:cex} we construct the counter-examples that prove Theorems~\ref{t:cex} and \ref{t:cex2} and Corollary~\ref{c:dpt}. 
\subsection*{Notation}
\begin{itemize}
\item $ \Leb^n$: the $n$-dimensional Lebesgue measure.
\item $\Haus^m$: the $m$-dimensional Hausdorff measure. 
\item $\mu \res E$: the restriction of the measure $\mu$ to the measurable set $E$. 
\item $\mathbf{1}_E$: the characteristic function of the set $E$. 
\item $\Omega$: an open set in $\R^d$ having uniformly Lipschitz continuous boundary.
\item $L^{\infty} (]0, T[   \times \partial \Omega): = L^{\infty} (]0, T[  \times \partial \Omega, \Leb^1 \otimes 
 \Haus^{d-1} )$, where we denote with $\otimes$ the (tensor) product of two measures.
\item $\div  b$: the distributional divergence of the vector field ${b: ]0, T[ \times \Omega \to \R^d}$, computed with respect to the $x \in \Omega$ variable only.
\item $\divtx B$: the standard ``full'' distributional divergence of the vector field $B$. In particular, when $B: ]0, T[ \times \Omega \to \R^{d+1}$, then  $\divtx B$ is the divergence computed with respect to the $(t, x) \in ]0, T[ \times \Omega$ variable .
\item $\nabla \varphi:$ the gradient of the smooth function ${\varphi: ]0, T[ \times \Omega \to \R^d}$, computed with respect to the $x \in \Omega$ variable only.
\item $\tr (b, \Sigma)$: the normal trace of the  vector field $b$ on the surface $\Sigma \subseteq \Omega$, as defined in~\cite{AmbrosioCrippaManiglia} (see also \S~\ref{ss:acm} in here).
\item $\tr b$: the normal trace of the vector field $b$ on $]0, T[ \times \partial \Omega$, defined as in~\S~\ref{ss:df}. 
\item $\mathcal M_{\infty}(\Lambda)$: the class of bounded, measure-divergence vector fields, namely the functions $B \in L^{\infty} (\Lambda; \R^N)$ such that the distributional divergence $\divtx B$ is a locally bounded Radon measure on the open set $\Lambda \subseteq \R^N$.
\item $|x|$: the Euclidian norm of the vector $x \in \R^d$. 
\item $\mathrm{supp} \, \rho$: the support of the smooth function $\rho: \R^N \to \R$.
\item $B_R (0)$: 
the ball of radius $R>0$ and center at $0$. 
\end{itemize}

\section{Normal traces of bounded, measure-divergence vector fields}
\label{s:pre}
\label{ss:acm}
We collect in this section some definitions and properties concerning weak traces of measure-divergence vector fields. Our presentation follows \cite[\S 3]{AmbrosioCrippaManiglia}. 

Given an open set $\Lambda \subseteq \R^N$, we denote by $\mathcal M_\infty (\Lambda)$ the family of bounded, measure-divergence vector fields, namely the functions $B \in L^{\infty} (\Lambda; \R^N)$ such that the distributional divergence $\divtx B$ is a locally bounded Radon measure on $\Lambda$. 
\begin{definition}\label{nomal_trace}
   Assume that $\Lambda \subseteq \R^N$ is a domain with uniformly Lipschitz continuous boundary. 
   Let $B \in \mathcal M_\infty (\Lambda)$, then the normal trace of $B$ on $\partial \Lambda$ can be defined as a distribution by the identity
   \begin{equation}\label{e:normal_trace}
     \langle\trace (B, \partial \Lambda), \varphi\rangle = \int_{\Lambda}\nabla\varphi\cdot B\,dx +\int_{\Lambda}\varphi\,d(
     \divtx B) \qquad \forall \varphi \in 
     \mathcal C^{\infty}_c (\R^N).  
   \end{equation}
\end{definition}
This definition is consistent with the Gauss-Green formula if the vector field $B$ is sufficiently smooth. In this case the distribution is induced by the integration of $B\cdot \vec n$ on $\partial \Lambda$, where $\vec n$ is the outward pointing, unit normal vector to $\partial \Lambda$. 
 
\begin{lemma}\label{proposition3.2_ACM}
  The distribution defined above is induced by an $L^\infty$ function on $\partial \Lambda$, which we can still call  $\trace (B, \partial \Lambda)$, with
  \begin{equation}\label{Linfty_bound}
    \|\trace (B, \partial \Lambda)\|_{L^\infty(\partial\Lambda)}\leq \|B\|_{L^\infty(\Lambda)}.
  \end{equation}
Moreover, if $\Sigma$ is a Borel set contained in $\partial \Lambda_1 \cap \partial \Lambda_2$  and if $\vec n_{1}=\vec n_{2}$ on $\Sigma$, then
\begin{equation}
\label{trace_on_sigma}
  \trace (B, \partial \Lambda_1)= \trace (B, \partial \Lambda_2) \quad \Haus^{N-1}-a.e. \; \text{on} \; \Sigma. 
\end{equation} 
\end{lemma}
Starting from the identity \eqref{trace_on_sigma}, it is possible to introduce the notion of normal trace on general bounded, oriented, Lipschitz continuous hypersurfaces $\Sigma\subseteq\R^N$. Indeed, once the orientation of $\vec n_\Sigma$ is fixed, we can find $\Lambda_1 \subseteq \R^N$  such that $\Sigma \subseteq\partial \Lambda_1$ and the normal vectors $n_\Sigma$ and $n_1$ coincide. Then we can define
\begin{equation}
  \trace^- (B, \Sigma): = \trace (B, \partial \Lambda_1). 
\end{equation}
Analogously, if $\Lambda_2 \subseteq \R^N$ is an open subset  such that $\Sigma\subseteq\partial\Lambda_2$,  and $\vec n_{2}=-\vec{n}_{\Sigma}$, we can define 
\begin{equation}
  \trace^+ (B, \Sigma): = - \trace (B, \partial \Lambda_2). 
\end{equation}
Note that we have the formula
\begin{equation} \label{divergence_on_boundary}
  (\divtx B) \res \Sigma= \Big( \trace^+(B,\Sigma)-\trace^-(B,\Sigma) \Big) \Haus^{N-1}\res\Sigma.
\end{equation}
In particular, $\trace^+$ and $\trace^-$ coincide $\Haus^{N-1}$-a.e.~on $\Sigma$ if and only if $\Sigma$ is negligible for the measure $\divtx B$.

We now go over some space continuity results established in~\cite[\S 3]{AmbrosioCrippaManiglia}. We first recall the definition of a family of graphs.  
\begin{definition}
\label{d:fg}
Let $I \subseteq \R$ be an open interval. A family of oriented surfaces $\{ \Sigma_r \}_{r \in I} \subseteq \R^N$ is a family of graphs if there are a bounded open set $D \subseteq \R^{N-1}$ and a Lipschitz function $f : D \to \R$ such that the following holds. There is a system of coordinates $(x_1, \cdots, x_N)$ such that, for every $r \in I$, 
$$
   \Sigma_r = \big\{ (x_1, \dots, x_N): \; f(x_1, \dots, x_{N-1}) - x_N =r  \big\}  
$$
and $\Sigma_r$ is oriented by the normal $(-  \nabla f, 1)/\sqrt{1 + |\nabla f|^2}$. 
\end{definition}
We now quote~\cite[Theorem 3.7]{AmbrosioCrippaManiglia}. 
\begin{theorem}
\label{t:wc} Let $B \in \mathcal M_{\infty}(\R^N)$ and let $\{ \Sigma_r \}_{r \in I}$ be a family of graphs as in Definition~\ref{d:fg}. 
Given $r_0 \in I$, we define the functions $\alpha_0, \alpha_r: D \to \R$  by setting 
$$
    \alpha_0 (x_1, \dots, x_{N-1}): = \tr^- (B, \Sigma_{r_0}) \big(x_1, \dots, x_{N-1}, f(x_1, \dots, x_{N-1})-r_0 \big)
$$
and 
$$
    \alpha_r (x_1, \dots, x_{N-1}): = \tr^+ (B, \Sigma_{r}) \big(x_1, \dots, x_{N-1}, f(x_1, \dots, x_{N-1})-r\big).
$$
Then we have 
$$
     \alpha_r \weaks\alpha_0 
     \quad 
     \text{weakly$^{\ast}$ in $L^{\infty} (D, \Leb^{N-1} \res D )$ as $r \to r_0^+$.}
$$
\end{theorem}
\section{Proof of Theorem~\ref{t:wp}}
\label{s:pwp}
\subsection{Preliminary results}
In this section we establish some results that are preliminary to the distributional formulation of problem~\eqref{e:clpb}.  
\begin{lemma}\label{l:preli2}
  Let $B$ be a locally bounded vector field on $\R^N$ and let $\{ \rho_\eps \}_{0< \eps <1}$ be a standard family of mollifiers satisfying $\mathrm{supp} \, \rho_\eps \subseteq B_\eps (0)$ for every $\eps \in ]0, 1[$. 

The divergence of $B$ is a locally finite measure if and only if for any $K$ compact in $\R^N$ there exists a positive constant $C$ such that the inequality 
\begin{equation}
\label{e:unifb}
  \|\divtx B \ast \rho_\eps\|_{L^1(K)}\leq C
\end{equation}\label{eq: preliminary_lemma}
holds uniformly in $\eps \in ]0, 1[$. 
\end{lemma}
\begin{proof}
  If $\divtx B$ is a locally finite measure the inequality \eqref{e:unifb} is satisfied on any compact $K$ for some constant $C$ independent from $\eps$.

On the other hand, the sequence $(\divtx B) \ast\rho_\eps = \divtx (B \ast\rho_\eps)$ converges to $\divtx B$ in the sense of distributions and the uniform bound \eqref{e:unifb} implies that we can extract a subsequence which converges weakly in the sense of measures.   
\end{proof}
\begin{lemma}\label{l:preli1}
 Let $\Lambda \subseteq \R^N$ be an open subset with uniformly Lipschitz continuous boundary and let $B$ belong to $\mathcal M_\infty (\Lambda)$. Then the vector field
 $$
     \tilde B(z) : = 
     \left\{
     \begin{array}{ll}
     B(z) & z \in \Lambda \\
     0 & \text{otherwise} \\
     \end{array}
     \right.
 $$
 belongs to $\mathcal M_{\infty} (\R^N)$.  
\end{lemma}
\begin{proof}
We only need to check that the distributional divergence of $\tilde B$ is a locally bounded Radon measure.  Given $\eps \in ]0, 1[$ we define the $\eps$-neighborhood of $\partial\Lambda$ as 
$$
   \displaystyle 
   \partial \Lambda_\eps= \{ z \in \R^N \: : \: \textrm{dist}(z,\partial\Lambda)<\eps \}.
$$
Any compact subset $K$ of $\R^N$ can be decomposed as follows:   
\begin{equation}
  K =\big( 
  K \cap
  ( \Lambda \setminus
  \partial \Lambda_\eps)
  \big) 
  \cup
  \big(K   \cap  
  \partial \Lambda_\eps \big) 
  \cup 
  \big(
  K \setminus(\Lambda \cup 
   \partial \Lambda_\eps) 
  \big). 
\end{equation}
Also, note that $\divtx(\tilde B\ast\rho_\eps)$ is zero on $K \setminus ( \Lambda \cup \partial \Lambda_\eps)$ and that its $L^1$ norm is uniformly bounded on  $K\cap (\Lambda \setminus \partial \Lambda_\eps)$. 
Moreover, 
\begin{equation}\label{L1_norm_sigma_eps}
  \begin{split}
    \int_{K \cap \partial \Lambda_\eps}
     |\div(\tilde B\ast\rho_\eps)| \, dz & = 
    \int_{K\cap  \partial \Lambda_\eps} 
    |\tilde B\ast\nabla\rho_\eps| \, dz \\
    & \leq \|\tilde  B \|_{L^\infty (\R^N)} 
    \| \nabla \rho_\eps \|_{L^1 
    ( \R^N)}
    \Leb^N (K \cap \partial
     \Lambda_\eps ). \\  
\end{split}
\end{equation}
We observe that $\|\tilde  B \|_{L^\infty (\R^N)} = \| B\|_{L^{\infty} (\Lambda)}$, that
$\Leb^N (K \cap \partial
     \Lambda_\eps ) \leq C_\ast \eps$ and that $\| \nabla \rho_\eps \|_{L^1 
    ( \R^N)} \leq C_{\ast \ast} / \eps$ for suitable constants $C_\ast >0$ and $C_{\ast \ast}>0$. Hence, 
 $$
     \int_{K \cap \partial \Lambda_\eps}
     |\div(\tilde B\ast\rho_\eps)| \, dz
     \leq \| B \|_{L^\infty (\Lambda)} C_\ast 
     C_{\ast \ast}
 $$     
and by relying on Lemma \ref{l:preli2} we conclude. 
\end{proof}
\subsection{Distributional formulation of problem~\eqref{e:clpb}}
\label{ss:df}
We can now discuss the distributional formulation of~\eqref{e:clpb}. The following result provides a distributional formulation of the normal trace of $b$ and $bu$ on $]0, T[ \times \partial \Omega$. 
\begin{lemma}
\label{l:wt}
Let $\Omega \subseteq \R^d$ be an open set with uniformly Lipschitz boundary and let $T>0$. Assume that $b \in L^{\infty} (]0, T[ \times \Omega; \R^d)$ is a vector 
field such that $\div b$ is a locally finite Radon measure on $]0, T[ \times \Omega$. Then there is a unique function, which in the following we denote by $\tr b$, that belongs to {$L^{\infty} (]0, T[ \times \partial \Omega) $} and satisfies 
\begin{equation}
\label{e:trb}
\begin{split}
   \int_0^T \int_{\partial \Omega} \tr b  \, \varphi \, d \Haus^{d-1} dt  - &
     \int_\Omega \varphi(0, x) \, dx =  
     \int_0^T  \int_\Omega  \partial_t \varphi + b \cdot  
     \nabla \varphi \, dx dt   \\ & + 
     \int_0^T  \int_\Omega  \varphi \, d( \div b ) 
     \quad  \forall \varphi \in \mathcal C^{\infty}_c 
     ([0, T[ \times \R^d)  . 
\end{split}
\end{equation}
Also, if $w \in L^{\infty}(]0, T[ \times \Omega)$ and $f \in 
L^{\infty}(]0, T[ \times \Omega)$ satisfy 
\begin{equation}
\label{e:snsdd}
   \int_0^T  \int_\Omega  w \big( \partial_t \eta+ b \cdot  
     \nabla \eta \big) dx dt + 
     \int_0^T \int_\Omega f \eta \, dx dt 
     =0   
     \quad  \forall \eta \in \mathcal C^{\infty}_c 
     (]0, T[ \times \Omega),    
 \end{equation}
then there are two uniquely determined functions, which in the following we denote by $\tr (bw) \in L^{\infty} (]0, T[ \times \partial \Omega) $ and $w_0 \in L^\infty (\Omega)$, that satisfy 
\begin{equation}
\label{e:trbu}
\begin{split}
   \int_0^T \int_{\partial \Omega} & \tr (bw)  \varphi \, d \Haus^{d-1} dt  - 
     \int_\Omega \varphi(0, \cdot) w_0 \, dx \\
     & =  
     \int_0^T  \int_\Omega w \big( \partial_t \varphi + b \cdot  
     \nabla \varphi \big) dx dt  + 
     \int_0^T \int_\Omega f \varphi \, dx dt 
     \quad 
     \forall \, \varphi \in \mathcal C^{\infty}_c 
     ([0, T[ \times \R^d). \\
     \end{split} 
     \end{equation} 
\end{lemma}     
Note that requirement~\eqref{e:snsdd} is nothing but the distributional formulation of the equation
\begin{equation}
\label{e:effe}
    \partial_t w + \div (bw ) = f
    \quad \text{in $]0, T[ \times \Omega$}.  
\end{equation}  
Also note that the existence of the function $w_0$ follows from Lemma~1.3.3 in \cite{Dafermos}, the new part is the existence of the function $\tr (bw)$.
\begin{proof}
We first establish the existence of a function $\tr b $ satisfying~\eqref{e:trb}. Note that the uniqueness of such a function follows from the arbitrariness of the test function $\varphi$. We define the vector field $B: \R^{d+1} \to \R^{d+1}$ by setting
\begin{equation}
\label{e:B}
    B(t, x) : = 
    \left\{
    \begin{array}{lll}
              (1, b) & (t, x) \in ]0, T[ \times \Omega \\ 
             \; \;  \; 0 & \text{elsewhere in $\R^{d+1}$} \\
    \end{array}
    \right.
\end{equation}
and we note that 
$
    \divtx B \big|_{]0, T[ \times \Omega} = \div b,  
$
therefore $B$ satisfies the hypotheses of Lemma~\ref{l:preli1} provided that $\Lambda: = ]0, T[ \times \Omega$. Hence,  $ B \in \mathcal M_{\infty} (\R^{d+1}).$
We apply Lemma~\ref{proposition3.2_ACM} and we observe that 
$
    \tr (B, \partial \Lambda) \big|_{\{ 0 \} \times \Omega} \equiv - 1. 
$
We can then conclude by setting 
$$
    \tr b : = \tr (B, \partial \Lambda) 
    \big|_{]0, T [ \times \partial \Omega}.
$$
The existence of the function $\tr (bw )$ satisfying~\eqref{e:trbu} can be established by setting 
\begin{equation}
\label{e:c}
    C(t, x) : = 
    \left\{
    \begin{array}{lll}
              (w, bw) & (t, x) \in ]0, T[ \times \Omega \\ 
             \; \;  \; 0 & \text{elsewhere in $\R^{d+1}$} \\
    \end{array}
    \right.
\end{equation}
and observing that condition~\eqref{e:snsdd} implies that
$
    \divtx C \big|_{]0, T[ \times \Omega} = f.   
$
We can then conclude by using the same argument as before, by setting 
\begin{equation}
\label{e:tc}
        w_0 : = - \tr (C, \partial \Lambda) \big|_{\{ 0 \} \times \Omega}
        \quad \text{and} \quad 
        \tr (bw) : = \tr (C, \partial \Lambda) \big|_{]0, T[ \times \partial \Omega}.
\end{equation}
\end{proof}
We now state the rigorous formulation of problem~\eqref{e:clpb}. 
\begin{definition}
\label{d:rf}
Let $\Omega \subseteq \R^d$ be an open set with uniformly Lipschitz boundary. Assume that $b \in L^{\infty} (]0, T[ \times \Omega; \R^d)$ is a vector field such that $\div b$ is a locally finite Radon measure on $]0, T[ \times \Omega.$ A distributional solution of~\eqref{e:clpb} is 
a function $u \in L^{\infty} (]0, T[ \times \Omega)$ such that
\begin{itemize}
\item[i)] $u$ satisfies equation~\eqref{e:snsdd} with $f\equiv0$;
\item[ii)] $w_0= \bar u$;
\item[iii)] $\tr (bu) = \bar g \tr b$ on the set $\Gamma^-$ which is defined as follows:
$$
   \Gamma^- : = \big\{ (t, x) \in ]0, T[ \times \partial \Omega: \; (\tr \, b) (t, x) < 0 \big\}. 
$$
\end{itemize}    
\end{definition}
Note that in Definition~\ref{d:rf} we only assume $f \equiv 0$ for the sake of simplicity. By removing the condition $f \equiv 0$ from point i) we obtain the distributional formulation of the initial-boundary value problem obtained by replacing the first line of~\eqref{e:clpb} with~\eqref{e:effe}. 
\subsection{Proof of Theorem~\ref{t:wp}}
First, we observe that the existence of a solution of~\eqref{e:clpb} is established in~\cite{CDS:hyp} by closely following an argument due to Boyer~\cite{Boyer}. The argument to establish uniqueness is organized in two main steps: in \S~\ref{sss:ren} we show that, under the hypotheses of Theorem~\ref{t:wp}, distributional solutions of~\eqref{e:clpb} enjoy renormalization properties. Next, in \S~\ref{sss:gro} we conclude by relying on a by now standard argument based on Gronwall Lemma. 
\subsubsection{Renormalization properties}
\label{sss:ren} We fix $u$ distributional solution of~\eqref{e:clpb} and we proceed according to the following steps.

{\sc Step 1:} we use the same argument as in Ambrosio~\cite{Ambrosio:trabv} to establish renormalization properties ``inside" the domain. More precisely, the Renormalization Theorem~\cite[Theorem 3.5]{Ambrosio:trabv} implies that the function $u^2$ satisfies
\begin{equation}
\label{e:a}
\begin{split}
   \int_0^T  \int_\Omega  u^2 \big( \partial_t \psi+ b \cdot  
     \nabla \psi \big) dx dt - \int_0^T \int_\Omega u^2 \div b \, \psi \, dx dt &
     = - \int_\Omega \bar u^2 \psi(0, \cdot) dx \\
     &
       \forall \psi \in \mathcal C^{\infty}_c 
     ([0, T[ \times \Omega).    \\
     \end{split}
\end{equation}

{\sc Step 2:} we establish a trace renormalization property. 

First, we observe that by combining hypothesis 3 in the statement of Theorem~\ref{t:wp} with Theorem 3.84 in the book by Ambrosio, Fusco and Pallara~\cite{AmbrosioFuscoPallara} we obtain that the vector field $B$ defined as in~\eqref{e:B} satisfies $B(t, \cdot) \in BV (\Omega_\ast)$ for every open and bounded set $\Omega_\ast \subseteq \R^d$ and for $\Leb^1$-a.e. $t \in ]0, T[$.  

Next, we recall that the proof of Lemma~\ref{l:wt} ensures that the vector field $uB$ belongs to $\mathcal M_\infty (\R^{d+1})$.  We can then apply~\cite[Theorem 4.2]{AmbrosioCrippaManiglia}, which implies the following trace renormalization property:
\begin{equation}
\label{e:trn}
           \tr (u^2 b) (t, x) = 
           \left\{
           \begin{array}{cc}
           \displaystyle{ 
           \left( \frac{\tr (ub) }{ \tr \, b }\right)^2 } \tr \, b & 
           \quad   \tr \, b (t, x) \neq 0 \\
           \phantom{ciao} \\
           0 & \qquad \tr \, b (t, x) = 0. \phantom{\int} \\
           \end{array}
           \right.
\end{equation}
Some remarks are here in order. First, to define $\tr (u^2 b)$ we recall~\eqref{e:a}, use Lemma~\ref{l:wt} and set 
\begin{equation}
\label{e:tru2}
    \tr (u^2 b) : = \tr (u^2 B, \partial \Lambda)\big|_{]0, T[ \times \partial \Omega},
\end{equation}
where $\Lambda = ]0, T[ \times \Omega$. 

Second, note that, strictly speaking, the statement of~\cite[Theorem 4.2]{AmbrosioCrippaManiglia} requires that the vector field $B$ has $BV$ regularity with respect to the $(t, x)$-variables, which in our case would imply some control on the time derivative of $b$. However, by examining the proof of~\cite[Theorem 4.2]{AmbrosioCrippaManiglia} and using the particular structure of the vector field $B$ one can see that only space regularity is needed to establish~\eqref{e:trn}.  

{\sc Step 3:} by combining~\eqref{e:a} with~\eqref{e:tru2} and recalling Lemma \ref{l:wt} we infer that  
\begin{equation}
\label{e:trbu2}
\begin{split}
   \int_0^T \int_{\partial \Omega} & \tr (u^2 b)  \varphi \, d \Haus^{d-1} dt  - 
     \int_\Omega \bar u^2 \varphi(0, x) \, dx 
      =  
     \int_0^T  \int_\Omega u^2 \big( \partial_t \varphi + b \cdot  
     \nabla \varphi \big) dx dt  \\ & - 
     \int_0^T \int_\Omega u^2 \div b \, \varphi \, dx dt 
     \qquad 
     \forall \, \varphi \in \mathcal C^{\infty}_c 
     ([0, T[ \times \R^d). \\
     \end{split} 
     \end{equation} 
\subsubsection{Conclusion of the proof of Theorem~\ref{t:wp}}
\label{sss:gro}
We conclude by following a by now standard argument, see for example the expository work by De Lellis~\cite[Proposition 1.6]{CDL:bourbaki}. We proceed according to the following steps.

{\sc Step A:} we observe that, the equation being linear, establishing that the distributional solution of~\eqref{e:clpb} is unique amounts to show that, if $\bar u \equiv 0$ and $\bar g \equiv 0$, then any distributional solution satisfies $u \equiv 0$. Also, note that in the remaining part of the proof we use~\cite[Lemma 1.3.3]{Dafermos} and we identify $u^2$ with its representative satisfying that the map $t \mapsto u^2(t, \cdot)$ is continuous in $L^{\infty}_{\loc} (\Omega)$ endowed with the weak$^{\ast}$ topology. 

{\sc Step B:} we fix $\bar t \in ]0, T[$ and we construct a sequence of test functions $\varphi_n$ as follows. First, we choose a function $h: [0, + \infty[ \to \R$ such that
\begin{equation}
\label{e:nu}
       h \in \mathcal C^{\infty}_c ([0, + \infty[),      
         \quad h \ge 0 \; \text{and} \; h' \leq 0 \; \text{everywhere in $[0, + \infty[$.}    
\end{equation} 
Next, we set 
\begin{equation}
\label{e:nu1}
   \nu(t, x) : = h \big( \| b \|_{L^{\infty} } |t- \bar t |  + | x | \big) 
\end{equation}
and we observe that $\nu$ satisfies 
\begin{equation}
\label{e:nu2}
    \partial_t \nu + b \cdot \nabla \nu \leq 
    \partial_t \nu + \| b \|_{L^{\infty}} |\nabla \nu| \leq 0 \quad 
    \text{$\Leb^{d+1}$-a.e. $(t, x) \in ]- \infty, \bar t[ \times \R^d$.}
\end{equation}
We then choose a sequence of cut-off functions $\chi_n \in \mathcal C^{\infty}_c ([0, + \infty[)$ satisfying 
\begin{equation}
\label{e:chi}
   \chi_n \equiv 1 \; \text{on $[0, \bar t \, ]$}, \; \;
   \chi_n \equiv 0 \; \text{on $[ \, \bar t + 1/n, + \infty[$}, \; 
     \;
   \chi_n' \leq 0 \; \text{everywhere on $[0, + \infty[$}.
\end{equation}
Finally, we set 
$$
    \varphi_n (t, x) : = \chi_n (t) \nu(t, x) \quad (t, x) \in [0, T[ \times \R^d
$$
and we observe that $\varphi_n \ge 0$ everywhere on $[0, + \infty[ \times \R^d$ and that $\varphi_n$ is compactly supported in~$[0, T[ \times \R^d$ provided that $n$ is sufficiently large. 

{\sc Step C:} we use $\varphi_n$ as a test function in~\eqref{e:trbu2}. First, we observe that by recalling that $\bar g \equiv 0$ and $\bar u \equiv 0$ and by using the renormalization property~\eqref{e:trn} we obtain that the left hand side of~\eqref{e:trbu2} is nonnegative, namely 
\begin{equation*}
\begin{split}
 0 & \leq 
 \int_0^T \! \! \! \int_\Omega u^2 \nu \chi_n' dx dt + 
  \int_0^T \! \! \! \int_\Omega \chi_n u^2 \big( \partial_t \nu + b \cdot  
     \nabla \nu \big) dx dt   - 
     \int_0^T\! \! \! \int_\Omega \div b \, u^2 \, \nu \chi_n  \, dx dt \,. \\
\end{split}
\end{equation*}
Next, we let $n \to + \infty$ and by recalling properties~\eqref{e:nu2} and~\eqref{e:chi} we obtain
$$
    \int_\Omega 
    \nu(\bar t, \cdot) \, u^2 (\bar t, \cdot) dx \leq \| \div b \|_{L^{\infty}} 
    \int_0^{\bar t} \int_\Omega \nu \, u^2 dx dt. 
$$
We can finally conclude by using Gronwall Lemma and the arbitrariness of the function $h$ in~\eqref{e:nu1}. This concludes the proof of Theorem~\ref{t:wp}. \hfill \qed

\subsection{Rigorous statement and proof of the space continuity property}
\label{ss:sc}
We provide a rigorous formulation of the analogue of the space continuity property stablished in the Sobolev case by Boyer in~\cite[\S~7.1]{Boyer}.
\begin{propos}
\label{p:sc}
Let $b$ be as in the statement of Theorem~\ref{t:wp}, 
$u~\in~L^{\infty} (]0, T[ \times \Omega)$ be the distributional solution of~\eqref{e:clpb} and $B \in \mathcal M_{\infty} (\R^{d+1}) $ be the same vector field as in~\eqref{e:B}.  Given a family of graphs $\{ \Sigma_r \}_{r \in I} \subseteq \R^d$ as in Definition~\ref{d:fg}, we fix $r_0 \in I$ and we define the functions $\gamma_0, \gamma_r: ]0,T[ \times D 
\to \R$ by setting 
$$
    \gamma_0 (t,x_1, \dots, x_{d-1}): = \tr^- (uB, ]0, T[ \times \Sigma_{r_0}) 
    \big(t,x_1, \dots, x_{d-1}, f(x_1, \dots, x_{d-1})-r_0 \big)
$$
and 
$$
    \gamma_r (t,x_1, \dots, x_{d-1}): = \tr^+ (uB, ]0, T[ \times \Sigma_{r}) 
    \big(t,x_1, \dots, x_{d-1}, f(x_1, \dots, x_{d-1})-r\big).
$$
Then  
\begin{equation}
\label{e:sc}
     \gamma_r  \to \gamma_0 
     \quad 
     \text{strongly in $L^1 (]0, T[ \times D )$ as $r \to r_0^+$.}
\end{equation}
\begin{proof}
The argument is organized in three steps. 

{\sc Step 1:} we make some preliminary considerations and introduce some notation. With a slight abuse of notation, 
we consider $b$ as a vector field defined on $\R^{d+1}$, set equal to zero out of $]0, T[ \times \Omega$.

By combining hypothesis 3 in the statement of Theorem~\ref{t:wp} with~\cite[Theorem 3.84]{AmbrosioFuscoPallara} we obtain that $b (t, \cdot) \in BV_\loc (\R^d)$ for $\Leb^1$-a.e. $t \in \R$. Hence, the classical theory of $BV$ functions (see for instance~\cite[Section 3.7]{AmbrosioFuscoPallara}) ensures that the outer and inner traces $b(t, \cdot)^+_{\Sigma_r}$ and $b(t, \cdot)^-_{\Sigma_r}$ are well-defined, vector valued functions for $\Leb^1$-a.e. $t \in \R$ and for every $r$. 

{\sc Step 2:} given $B$ as in \eqref{e:B}, we define the functions $\beta_0, \beta_r: ]0,T[ \times D \to \R$ by setting 
$$
    \beta_0 (t, x_1, \dots, x_{d-1}): = \tr^- (B, ]0, T[ \times \Sigma_{r_0}) 
    \big(t, x_1, \dots, x_{d-1}, f(x_1, \dots, x_{d-1})-r_0 \big)
$$
and 
$$
    \beta_r (t, x_1, \dots, x_{d-1}): = \tr^+ (B, ]0, T[ \times \Sigma_{r}) 
    \big(t, x_1, \dots, x_{d-1}, f(x_1, \dots, x_{d-1})-r\big).
$$
We claim that 
\begin{equation}
\label{e:beta}
\beta_r  \to \beta_0 
     \quad 
     \text{strongly in $L^1 (]0, T[ \times D )$ as $r \to r_0^+$.}
\end{equation}
To establish~\eqref{e:beta}, we first observe that by using~\cite[Theorem 3.88]{AmbrosioFuscoPallara} and an approximation argument one can show that for every 
$r \in I$ we have
$$
    \beta_r =  
    b^+_{\Sigma_r} \cdot  \vec m, 
    \quad \text{and} \quad
    \beta_0 = b^-_{\Sigma_0} \cdot \vec m \quad 
    \text{for $\Leb^d$-a.e. 
     $(t, x) \in ]0, T[ \times D$}.
$$   
In the previous expression, $\vec m = (- \nabla f, 1)/ \sqrt{1 + |\nabla f|^2}$ is the unit normal vector defining the orientation of $\Sigma_r$. Also, by again combining~\cite[Theorem 3.88]{AmbrosioFuscoPallara} with an approximation argument we get that 
$$
    \int_0^T \int_D |\beta_r - \beta_0 | dx_1 \dots d_{x_{d-1}} dt \leq 
    \int_0^T |D b (t, \cdot) | ( S ) dt,  
$$ 
which implies~\eqref{e:beta}. In the previous expression, $|Db(t, \cdot)|$ denotes the total variation of the distributional derivative of $b(t, \cdot)$, and $S$ is the set
$$
\begin{aligned}
S:= \Big\{ (x_1, \ldots , x_{d-1}, x_d) \, : \,
&(x_1, \ldots , x_{d-1}) \in D \; \text{ and } \; \\
&f(x_1, \ldots , x_{d-1}) - r < x_d < f(x_1, \ldots , x_{d-1}) - r_0 \Big\} \,.
\end{aligned}
$$

{\sc Step 3:} we conclude the proof of Proposition~\ref{p:sc}. 
First, we observe that due to Theorem~\ref{t:wc} we have that
\begin{equation}
\label{e:wc2}
     \gamma_r \rightharpoonup \gamma_0 \text{ weakly in $L^2 (]0, T[ \times D )$ as $r \to r_0^+$.}    
\end{equation}
Next, we recall that $\gamma_r$ is the normal trace of $u B$ and that $\beta_r$ is the trace of $B$, so that by applying~\cite[Theorem 4.2]{AmbrosioCrippaManiglia} we get 
\begin{equation}
\label{e:square}
    \gamma^2_r  = \beta_r \tr^+ (u^2 B, ]0, T[ \times \Sigma_{r}) 
    \quad \text{and} \quad 
    \gamma^2_0  = \beta_0 \tr^- (u^2 B, ]0, T[ \times \Sigma_{r_0}). 
\end{equation}
By combining~\eqref{e:beta} with the uniform bound $\| \beta_r\|_{L^\infty} \leq \| b \|_{L^\infty}$ we infer that $\beta_r \to \beta_0$ strongly in $L^2 (]0, T[ \times D)$. Then we apply Theorem~\ref{t:wc} to $\tr^+ (u^2 B, ]0, T[ \times \Sigma_{r})$ and hence by recalling~\eqref{e:square} we conclude that
\begin{equation}
\label{e:wc3}
    \gamma^2_r \rightharpoonup \gamma^2_0 \text{ weakly in $L^2 (]0, T[ \times D )$ as $r \to r_0^+$.}  
\end{equation}
By using~\eqref{e:wc2}, we get that~\eqref{e:wc3} implies that $\gamma_r \to \gamma_0$ strongly in $L^2 (]0, T[ \times D )$ and from this we eventually get~\eqref{e:sc}. 
\end{proof}
\end{propos}

\section{Counter-examples}
\label{s:cex}
\subsection{Some notation and a preliminary result}
\label{ss:cex:not}
For the reader's convenience, we collect here some notation
we use in this section.
\begin{itemize}
\item Throughout all \S~\ref{s:cex}, $\Omega$ denotes the set 
$]0, + \infty[ \times \R^2$. 
\item We use the notation $(r, y) \in ]0, + \infty[ \times \R^2$ or, if needed, the notation $(r, y_1, y_2)~\in~]0, + \infty[ ~\times ~\R ~\times~\R$
to denote points in $\Omega$.  
\item $\mathrm{div}$ denotes the divergence computed with respect to the $(r, y)$-variable.
\item $\mathrm{Div}$ denotes the divergence computed with respect to the $(t, r, y)$-variable.
\item $\divy$ denotes the divergence computed with respect to the $y$ variable only. 
\item We decompose $]0, 1[ \times \Omega$ as 
$]0, 1[ \times \Omega = \Lambda^+ \cup \Lambda^- \cup \mathcal S$, where 
\begin{equation}
\label{e:lambdap}
   \Lambda^+ : = \{ (t, r, y) \in ]0, 1[ \times \Omega: \; r > t  \}
\end{equation}
and 
\begin{equation}
\label{e:lambdam}
   \Lambda^- : = \{ (t, r, y) \in ]0, 1[ \times \Omega: \; r < t  \},
\end{equation}   
while $\mathcal S$ is the surface 
\begin{equation}
\label{e:sigma}
         \mathcal S : = 
         \{ (t, r, y) \in ]0, 1[ \times \Omega: \; r = t  \}. 
\end{equation}
\end{itemize}
We also observe that, thanks to~\cite[Lemma 1.3.3]{Dafermos}, up to a redefinition of $u(t,x)$ in a negligible set of times, we can assume that the map $t \mapsto u(t, \cdot)$ is continuous from $]0, 1[$ in $L^{\infty} (\Omega)$ endowed with the weak-$\ast$ topology, and in particular 
$$
    u(t, \cdot) \weaks u_0 \quad \text{in $L^{\infty} (\Omega)$
    as $t \to 0^+$,} 
$$
where $u_0$ the value attained by $u$ at $t=0$, as in Lemma~\ref{l:wt}. 
%\begin{lemma}
%\label{l:easy}
%Let $\Omega: = ]0, + \infty[ \times \R^2$ and assume that $b \in L^{\infty} (]0, 1[ \times \Omega; \R^3)$ and $u \in L^{\infty} (]0, 1[ \times \Omega)$ satisfy
%$$
%    \partial_t u + \div (bu) =0 \quad 
%    \text{in $]0, 1[ \times \Omega$},
%$$ 
%and let $u(t, \cdot)$ be the representative of $u$ such that the map $t \mapsto u(t, \cdot)$ is continuous from $]0, 1[$ in $L^{\infty} (\Omega)$ endowed with the weak-$\ast$ topology. Then 
%$$
%    u(t, \cdot) \weaks u_0 \quad \text{in $L^{\infty} (\Omega)$
%    as $t \to 0^+$}. 
%$$
%\end{lemma}
%\begin{proof}
%
%
%We extend $b$ and $u$ to $]- \infty, T[ \times \Omega$ by setting 
%\begin{equation}
%\label{e:tbu}
%   \tilde u(t, x) =
%   \left\{
%   \begin{array}{ll}
%   u(t, x) & t >0 \\
%   \bar u (x) & t <0. 
%   \end{array}
%   \right.
%   \quad \text{and} \quad 
%   \tilde b(t, x) =
%   \left\{
%   \begin{array}{ll}
%   b(t, x) & t >0 \\
%   0 & t <0. 
%   \end{array}
%   \right.
%\end{equation}
%Note that $\tilde u$ is a distributional solution of $\partial_t \tilde u~+~\div (\tilde b \tilde u)~=~0$ in $]- \infty, 1[ \times \Omega$. We then apply~\cite[Lemma 1.3.3]{Dafermos}, which states that there is a representative of $\tilde u$ such that $t \mapsto \tilde u(t, \cdot)$ is weakly-$^\ast$ continuous. This concludes the proof.
%\end{proof}
\subsection{Proof of Theorem~\ref{t:cex}}
\label{s:proof1} 
The proof is organized in three steps. 

{\sc Step 1:} we recall an intriguing example due to Depauw~\cite{Depauw} which is pivotal to our construction. In~\cite{Depauw}, Depauw explicitly exhibits a vector field $c: ]0, + \infty[ \times \R^2 \to \R^2$ satisfying the following properties:
\begin{itemize}
\item[a)] $c \in L^{\infty} (]0, + \infty[ \times \R^2; \R^2)$.  
\item[b)] For every $r>0$, $c(r, \cdot )$ is piecewise smooth and, for almost every $y \in \mathbb R^2$, the characteristic curve trough $y$ is well defined.
\item[c)] $ \divy c \equiv 0$. 
\item[d)] $c \in  L^1_\mathrm{loc} \big(       ]0, 1 [ ; \bvloc( \R^2; \mathbb R^2) \big)$, but $c \notin  L^1 \big(
       [0, 1 [ ; \bvloc( \R^2; \mathbb R^2) \big)$. Namely, the $BV$ regularity deteriorates as $r \to 0^+$. 
\item[e)] The Cauchy problem
\begin{equation}
\label{e:depauw}
    \left\{
        \begin{array}{lll}
              \partial_r w+ \divy (c w ) = 0 & \text{on  $]0, 1[ \times \R^2$} \\
               w = 0 & \text{at $r=0$} \phantom{\int} \\
        \end{array}
       \right.  
\end{equation}
admits a nontrivial bounded solution, which in the following we denote by $v(r, y)$.
\end{itemize}
{\sc Step 2:} we exhibit a vector field $b$ satisfying properties i), $\dots$, v) in the statement of Theorem~\ref{t:cex2}. We recall that the sets $\Lambda^+$, $\Lambda^-$ and $\mathcal S$ are defined by~\eqref{e:lambdap},~\eqref{e:lambdam} and~\eqref{e:sigma}, respectively. We define the vector field $b: ]0, 1[ \times \Omega \to \R^3$ by setting 
\begin{equation}
\label{e:b}
         b(t, r, y) : = 
         \left\{
          \begin{array}{ll}
                     \displaystyle{\big(1, c(r, y) \big) } & \text{in $\Lambda^-$}  \\
                       \displaystyle{\big(1, 0 \big) } & \text{in $\Lambda^+$ }   \\
          \end{array} 
         \right.
\end{equation}
In the previous expression, $c$ is Depauw's vector field as in {\sc Step 1}. By relying on properties a), c) and d) in {\sc Step 1} one can show that $b$ satisfies properties i), ii), iii) in the statement of Theorem~\ref{t:cex}. 

Next, we recall that the initial-boundary value problem~\eqref{e:cex:ibvp} admits the trivial solution $u \equiv 0$ and that the linear combination of solutions is again a solution. Hence, establishing property v) in the statement of Theorem~\ref{t:cex}
amounts to exhibit a nontrivial solution of~\eqref{e:cex:ibvp}. 
We define the function $u$ by setting 
\begin{equation}
\label{e:u2}
   u(t, r, y) : = 
   \left\{
   \begin{array}{ll} v(r, y)  & \text{in $\Lambda^-$} \\
    0  & \text{in $\Lambda^+$, }  \\
  \end{array} 
   \right.
\end{equation}
where $v$ is the same function as in {\sc Step 1}. 

{\sc Step 3:} we show that the function $u$ is a distributional solution of~\eqref{e:cex:ibvp}. We set $C:=(u, bu)$ and we observe that by construction $\divtx C \equiv 0$ on $\Lambda^+$. Also, property e) in {\sc Step 1} implies that $\divtx C \equiv 0$ on $\Lambda^-$. Finally, by recalling~\eqref{divergence_on_boundary} we infer that $\divtx C \res \mathcal S =0$ since the normal trace is $0$ on both sides. 

We are left to show that the initial and boundary data are attained. First, we observe that $u(t, \cdot) \weaks 0$ as $t \to 0^+$ and hence $u_0 \equiv 0$ by the weak continuity of $u$ with respect to time. Next, we fix an open and bounded set $D \subseteq \R^2$ and we define the family of graphs $\{ \Sigma_r \}_{r \in ]0, 1[} \subseteq ]0, 1[ \times \Omega$  by setting 
\begin{equation}
\label{e:sigmar}
 \Sigma_r : = \big\{ 
    (t, r, y_1, y_2): \; t \in ]0, 1[ \; \text{and} \;  (y_1, y_2) \in D \big\}.
\end{equation}
The orientation is given by the vector $(0, -1, 0, 0)$. We point out that requirement e) in {\sc Step 1} implies that $v(r, \cdot) \weaks 0$ as $r \to 0^+$. Hence, by recalling that $b$ is given by~\eqref{e:b}, we obtain that $\tr^+ (C, \Sigma_r) \weaks 0$ as $r \to 0^+$. By recalling Theorem~\ref{t:wc} and Definition~\ref{d:rf}, we infer that $\tr (bu) \equiv 0$ on $]0, 1[ \times \partial \Omega$. 
This concludes the proof of Theorem~\ref{t:cex}. \hfill \qed

\subsection{Proof of Theorem~\ref{t:cex2}}
The proof is divided in four main steps:
\begin{enumerate}
\item in \S~\ref{sss:betak} we construct the auxiliary vector field $\beta_k$, which will serve as a ``building block" for the construction of the vector field $b$;
\item in \S~\ref{sss:b} we define the vector field $b$;
\item in \S~\ref{sss:reg} we show that $b$ satisfies properties iii) and iv) in the statement of Theorem~\ref{t:cex2};
\item finally, in \S~\ref{sss:nt} we exhibit a non trivial solution of the initial-boundary value problem~\eqref{e:cex:ibvp2}. Since the problem is linear, any linear combination of solutions is also a solution and hence the existence of a nontrivial solution implies the existence of infinitely many different solutions.  
\end{enumerate}
\subsubsection{Construction of the vector field $\beta_k$}
\label{sss:betak}
We fix $k \in \mathbb N$ and we construct the vector field $\beta_k$, which is defined on the cell 
$$
    (r, y_1, y_2) \in \, ]0, 4 \cdot 2^{-k} [ \times ]0, 4 \cdot 2^{-k} [ \times  ]0, 4 \cdot 2^{-k} [. 
$$
We split the $r$-interval $]0, 4 \cdot 2^{-k}[$ into four equal sub-intervals and we proceed according to the following steps. \\
{\sc Step 1:} if $r \in ] 0, 2^{-k} [$, we consider a ``three-colors chessboard" in the $(y_1,y_2)$-variables at scale $2^{-k}$ as in Figure~\ref{f:s1}, left part. The vector field $\beta_k$ attains the values $(1, 0, 0)$, $(-5, 0, 0)$ and $(0, 0, 0)$ on dashed, black and white squares, respectively. Note that $\beta_k$ satisfies 
\begin{equation}
\label{e:dfs1}
    \div \beta_k \equiv 0 \quad \text{on $]0,  2^{-k} [ \times ]0, 4 \cdot 2^{-k} [ \times  ]0, 4 \cdot 2^{-k} [$}
\end{equation}
since $\beta_k$ is piecewise constant and tangent at its discontinuity surfaces. 

Here is the rigorous definition of $\beta_k$: we set 
$$
    D_k : = \bigcup _{n, m=0, 1}
    ]    (2 n) 2^{-k} , (2n+1 )2^{-k}  [ \times  ]    (2 m)2^{-k} , (2m+1 )2^{-k}  [  
$$
and 
\begin{equation}
\label{e:dk}
    B_k : = \bigcup _{n, m=0, 1}
    ]    (2 n+1) 2^{-k} , (2n+2 )2^{-k}  [ \times  ]    (2 m+1)2^{-k} , (2m+2 )2^{-k}  [.
\end{equation}
Note that $D_k$ and $B_k$ are represented in the left part of Figure~\ref{f:s1} by dashed and black regions, respectively. 
Next, we define 
\begin{equation}
\label{e:betak1}
    \beta_k(r, y_1, y_2):=
    \left\{
    \begin{array}{lll}
    (1, 0, 0) & 
    \text{if $(y_1, y_2) \in D_k$} \\    
    \phantom{\int} \\
    (-5, 0, 0) & 
    \text{if $(y_1, y_2) \in B_k$} \\    
    \phantom{\int}  \\
     (0, 0, 0) & \text{elsewhere on $\big] 0, 4 \cdot 2^{-k} \big[  
    \times \big] 0, 4 \cdot 2^{-k} \big[   $} \\
     \end{array}
    \right.
\end{equation}
\begin{figure}
\begin{center}
\caption{The vector field $\beta_k(r, y_1, y_2)$ for different values of $r$: the dashed, black and white  squares are the region where $\beta_k$ attains the values $(1, 0, 0)$,  $(-5, 0, 0)$ and $(0, 0, 0)$, respectively.} 
\psfrag{a}{\small{$y_2$}} \psfrag{b}{\small{$y_1$}}
\psfrag{A}{\small{$r \in \, ]0, 2^{-k}[$}} 
\psfrag{B}{\small{$r = 2\cdot 2^{-k}$}} 
\psfrag{C}{\small{$r \in \, ]3 \cdot 2^{-k}, 4 \cdot 2^{-k}[$}} 
\label{f:s1}
\bigskip
\includegraphics[scale=0.4]{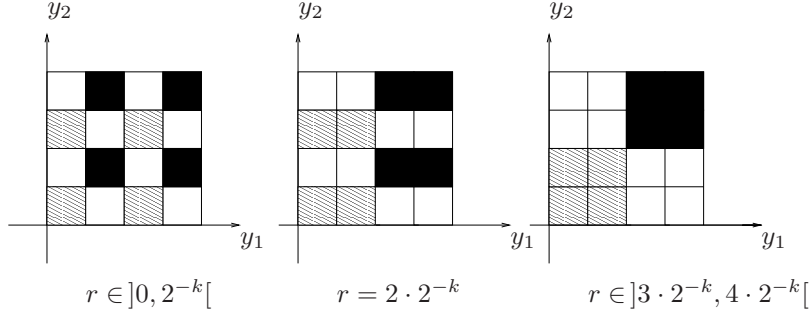}
\end{center}
\end{figure}
{\sc Step 2:} if $r \in ]  2^{-k}, 2 \cdot 2^{-k} [$, then the heuristic idea to define $\beta_k$ is that we want to (i) horizontally leftward slide the rightmost dashed squares and (ii) horizontally rightward slide the leftmost black squares. The final goal is that at $r = 2 \cdot 2^{-k}$ we have reached the configuration of the vector field described in Figure~\ref{f:s1}, center part.  The nontrivial issue is that we also require that 
\begin{equation}
\label{e:dfs2}
    \div \beta_k \equiv 0 \quad \text{on $]  0, 2 \cdot 2^{-k} [ \times ]0, 4 \cdot 2^{-k} [ \times  ]0, 4 \cdot 2^{-k} [$}. 
\end{equation}
To achieve~\eqref{e:dfs2}, we employ the construction illustrated in Figure~\ref{f:s2}: the vector field $\beta_k$ attains the value $(1, 0, 0)$ on the horizontal part of the dashed region, the value $(1, -1, 0)$ on the inclined part of the dashed region and the value $(0, 0, 0)$ elsewhere. Note that~\eqref{e:dfs2} is satisfied because $\beta_k$ is piecewise constant and it is tangent at its discontinuity surfaces on the interval $r \in ]   2^{-k}, 2 \cdot 2^{-k} [ $. We conclude by recalling~\eqref{e:dfs1} and by observing that the normal trace is continuous at the discontinuity surface $r = 2^{-k}$ and hence no divergence is created there. 

Here is the rigorous definition of $\beta_k$ for $r \in ]   2^{-k}, 2 \cdot 2^{-k} [$:
\begin{equation}
\label{e:betak2}
    \beta_k(r, y_1, y_2):=
    \left\{
    \begin{array}{lll}
    (1, 0, 0) & 
    \text{if $(y_1, y_2) \in ]0, 2^{-k}[ \times ]0, 2^{-k}[$} \\    
    \phantom{\int} & \text{or $(y_1, y_2) \in ]0, 2^{-k}[ \times ]2 \cdot 2^{-k}, 
    3 \cdot 2^{-k}[ $ }\\
     \phantom{\int} \\
    (1, -1, 0) & 
    \text{if $ - r + 3 \cdot 2^{-k} < y_1 <  - r + 4 \cdot 2^{-k}$} \\   
    \phantom{\int} & \text{and $y_2 \in ]0,  2^{-k}[$ or 
    $y_2 \in ]2 \cdot 2^{-k}, 
    3 \cdot 2^{-k}[$ }\\
     \phantom{\int} \\ 
        (-5, 0, 0) & 
    \text{if $(y_1, y_2) \in ]3 \cdot 2^{-k}, 4 \cdot 2^{-k}[ \times ] 2^{-k}, 2 \cdot 2^{-k}[$} \\    
    \phantom{\int} & \text{or $(y_1, y_2) \in  ]3 \cdot 2^{-k}, 4 \cdot 2^{-k}[
     \times ]3 \cdot 2^{-k}, 4 \cdot 2^{-k}[$ } \\   
     \phantom{\int} \\ 
     (-5, -5, 0) & 
    \text{if $  r   < y_1 <   r +   2^{-k}$} \\   
    \phantom{\int} & \text{and $y_2 \in ] 2^{-k}, 2 \cdot 2^{-k}[$ or 
    $y_2 \in ]3 \cdot 2^{-k}, 4 \cdot 2^{-k}[ $}\\
    \phantom{\int}  \\
     (0, 0, 0) & \text{elsewhere on $\big] 0, 4 \cdot 2^{-k} \big[  
    \times \big] 0, 4 \cdot 2^{-k} \big[   $} \\
     \end{array}
    \right.
\end{equation}
\begin{figure}
\begin{center}
\caption{The vector field $\beta_k(r, y_1, y_2)$ for $y_2 \in ]0, 2^{-k}[$ and ${y_2 \in] 2 \cdot 2^{-k}, 3 \cdot 2^{-k}[}$. The field $\beta_k$ attains the value $(1, -1, 0)$ in the inclined dashed region.} 
\psfrag{a}{\small{$r = 2 \cdot 2^{-k}$}} \psfrag{r}{\small{$r$}}
\psfrag{y2}{\small{$y_1$}} 
\label{f:s2}
\bigskip
\includegraphics[scale=0.3]{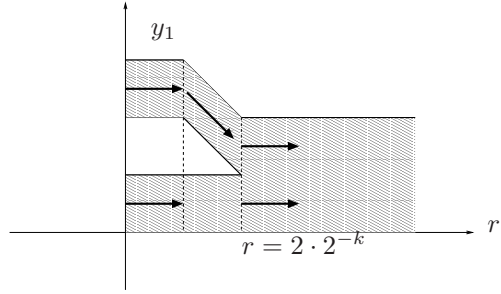}
\end{center}
\end{figure}
{\sc Step 3:} if $r \in ]  2 \cdot 2^{-k}, 3 \cdot 2^{-k} [$, the heuristic idea is defining $\beta_k$ in such a way that (i) we  push up the lower black region in Figure~\ref{f:s1}, central part, (ii) we pull down the upper dashed region in Figure~\ref{f:s1}, central part and (iii) we satisfy the requirement that $\beta_k$ is divergence-free. This is done by basically using the same construction as in {\sc Step 2}. Note that at $r = 3 \cdot 2^{-k}$ we have reached the configuration described in Figure~\ref{f:s1}, right part. 

Here is the rigorous definition of $\beta_k$ for $r \in ]2 \cdot 2^{-k}, 3 \cdot 2^{-k}[$:
$$
    \beta_k(r, y_1, y_2):=
    \left\{
    \begin{array}{lll}
    (1, 0, 0) & 
    \text{if $(y_1, y_2) \in ]0, 2 \cdot 2^{-k}[ \times ]0, 2^{-k}[$} \\    
     \phantom{\int} \\
    (1, 0, -1) & 
    \text{if $y_1 \in ]0, 2 \cdot  2^{-k}[$} \\
    \phantom{\int} & \text{and  $ - r + 4 \cdot 2^{-k} < y_2 <  - r + 5 \cdot 2^{-k}$} \\ 
        \phantom{\int} \\ 
        (-5, 0, 0) & 
    \text{if $(y_1, y_2) \in 
    ]2 \cdot 2^{-k}, 4 \cdot 2^{-k}[ \times  ]3 \cdot 2^{-k}, 4 \cdot 2^{-k}[$ } \\   
     \phantom{\int} \\ 
     (-5, 0, -5) & 
    \text{if $  y_1 \in ]2 \cdot 2^{-k}, 4 \cdot 2^{-k}[ $} \\   
    \phantom{\int} & \text{and $ r - 2^{-k} < y_2 < r $}\\
    \phantom{\int}  \\
     (0, 0, 0) & \text{elsewhere on $\big] 0, 4 \cdot 2^{-k} \big[  
    \times \big] 0, 4 \cdot 2^{-k} \big[   $} \\
     \end{array}
    \right.
$$
{\sc Step 4:} if $r \in ]  3 \cdot 2^{-k}, 4 \cdot 2^{-k} [$, then we consider the ``three colors chessboard" in the $(y_1,y_2)$-variables at scale $2 \cdot 2^{-k}$ illustrated in Figure~\ref{f:s1}, right part. The vector field $\beta_k$ attains the value $(1, 0, 0)$, $(-5, 0, 0)$ and $(0, 0, 0)$ on dashed, black and white regions, respectively. 

Here is the rigorous definition of $\beta_k$ for $r \in ]3 \cdot 2^{-k}, 4 \cdot 2^{-k}[$: 
$$
    \beta_k(r, y_1, y_2):=
    \left\{
    \begin{array}{lll}
    (1, 0, 0) & 
    \text{if $(y_1, y_2) \in ]0, 2 \cdot 2^{-k}[ \times ]0, 2 \cdot 2^{-k}[$} \\    
     \phantom{\int} \\
        (-5, 0, 0) & 
    \text{if $(y_1, y_2) \in 
    ]2 \cdot 2^{-k}, 4 \cdot 2^{-k}[ \times  ]2 \cdot 2^{-k}, 4 \cdot 2^{-k}[$ } \\   
    \phantom{\int}  \\
     (0, 0, 0) & \text{elsewhere on $\big] 0, 4 \cdot 2^{-k} \big[  
    \times \big] 0, 4 \cdot 2^{-k} \big[   $} \\
     \end{array}
    \right.
$$
Note that by construction
\begin{equation}
\label{e:dfs4}
          \div \beta_k \equiv 0 \quad \text{on $]0,  4 \cdot  2^{-k} [ \times ]0, 4 \cdot 2^{-k} [ \times  ]0, 4 \cdot 2^{-k} [$}.
\end{equation}
\subsubsection{Construction of the vector field $b$}
\label{sss:b}
We now define the vector field $b$ by using as a ``building block" the vector field $\beta_k$ defined in \S~\ref{sss:betak}.  We proceed in three steps. \\
{\sc Step A:} we extend $\beta_k$ to $]0, 2^{2-k}[ \times \R^2$ by imposing that it is $2^{2-k}$-periodic in both $y_1$ and $y_2$, namely we set
\begin{equation}
\label{e:per}
    \beta_k (r, y_1 + m 2^{2-k}, y_2 + n 2^{2-k}  ): = \beta_k (r, y_1, y_2)
\end{equation}
for every $m, n \in \mathbb Z$
and 
    $(y_1, y_2) \in \big] 0,  2^{2-k} \big[  
    \times \big] 0, 2^{2-k} \big[$. 
We recall~\eqref{e:dfs4} and we observe that $\beta_k$ is tangent at the surfaces $y_1= m 2^{2-k}$ and $y_2 = n 2^{2-k}$, $m, n \in \mathbb Z$. We therefore get 
\begin{equation}
\label{e:dfsa}
          \div \beta_k \equiv 0 \quad \text{on $]0,    2^{2-k} [ \times \R^2$}.
\end{equation}
{\sc Step B:} we define the vector field $b (t, r, y_1, y_2)$ on the set $\Lambda^-$ defined by~\eqref{e:lambdam}. To this end, we introduce the decomposition 
\begin{equation}
\label{e:dec}
   ]0, 1[: = \mathcal N \cup \bigcup_{k =3}^{\infty}  I_k ,
\end{equation}
where $\mathcal N$ is an $\Leb^1$-negligible set and
$$
    I_k : = 
    \left\{
    \begin{array}{lll}
            ]1/2, 1 [  & \text{if $k =3$}  \\
            \phantom{\int} \\
          \displaystyle{  \Big]  1- \sum_{j=3}^{k} 2^{2-j}, 1 - \sum_{j=3}^{k-1} 2^{2-j} 
          \Big[ }
           & \text{if  $k \ge 4$}. \\
   \end{array}
   \right.
$$   
We then set 
\begin{equation}
\label{e:brt}
         b(t, r, y_1, y_2) : = 
         \beta_k (r + 1- \sum_{j=3}^{k} 2^{2-j}, y_1, y_2) \quad \text{in $\Lambda^-$, when $r \in I_k$.} 
\end{equation}
Some remarks are here in order. First, to illustrate the heuristic idea underlying definition~\eqref{e:brt} we focus on the behavior at time $t=1$. The vector field $b(1, \cdot)$ behaves like $\beta_3$ on the interval $r \in ]1/2, 1[$, like $\beta_4$ on the interval $r \in ]1/4, 1/2[$, like $\beta_5$ on the interval $r \in ]1/8, 1/4[$, and so on. In other words, as $r \to 0^+$ the $r$-component of vector field $b$ oscillates between the values $1$, $-5$ and $0$ on a finer and finer ``three-colors chessboard". 

Second, we recall~\eqref{e:dfsa} and we observe that the vector field is continuous at the surfaces $r = 1 - \sum_{j=3}^{k-1} 2^{2-j}$, $k \ge 4$. Hence, 
\begin{equation}
\label{e:dfsb}
         \div b \equiv 0 \quad \text{on $\Lambda^-$}
          \end{equation} 
where the set $\Lambda^-$ is the same as in~\eqref{e:lambdam}. 

{\sc Step C:}  we define the vector field $b (t, r, y_1, y_2)$ on the set $\Lambda^+$ defined by~\eqref{e:lambdap}. To this end, we use the same decomposition of the unit interval as in~\eqref{e:dec} and we set
\begin{equation}
\label{e:btr}
     b(t, r, y_1, y_2) : = (1, 0, 0) \mathbf{1}_{D(t)} +
     (-5, 0, 0) \mathbf{1}_{B(t)} \quad \text{in $\Lambda^+$, when $t \in I_k$}.
\end{equation}
In the previous expression, $\mathbf{1}_{D(t)}$ denotes the characteristic function of the set 
$$
    D(t)  : = \Big\{ (r, y_1, y_2): \beta_k (t + 1- \sum_{j=3}^{k} 2^{2-j}, y_1, y_2) \cdot (1, 0, 0) =1     \Big\}
$$
and $\mathbf{1}_{B(t)}$ is the characteristic function of the set 
$$
    B(t)  : = \Big\{ (r, y_1, y_2): 
    \beta_k (t +  1- \sum_{j=3}^{k} 2^{2-j}, y_1, y_2) \cdot (1, 0, 0) =-5    \Big\}
$$
Hence, we have  
$
         \div b \equiv 0$ on $\Lambda^+$
since $b$ is piecewise constant and tangent at the discontinuity surfaces.
Next, we observe that by construction for any $t \in ]0,1[$ the normal trace of $b$ is continuous at the surface $\{ (r, y_1, y_2) : \; r=t \}\subseteq \Omega$.  By recalling~\eqref{e:dfsb}, we arrive at 
\begin{equation}
\label{e:dfsc}
        \div b \equiv 0 \quad \text{on $]0, 1[ \times \Omega$}. 
\end{equation}
\subsubsection{Proof of the regularity and trace properties} 
\label{sss:reg}
We first observe that, for every open and bounded set $\Omega_\ast$ such that $\bar \Omega_\ast \subseteq \Omega$, we have $b \in L^1 ( [0, T[; BV (\Omega_\ast ; \R^3))$ since it is piecewise constant on $\Omega_\ast$. However, note that the $BV$ regularity degenerates at the boundary $r=0$. 

Next, we prove that $\tr b \equiv 1$: we use the family of graphs $\{ \Sigma_r \}_{r \in ]0, 1[}$ defined as in~\eqref{e:sigmar}. By relying on Theorem~\ref{t:wc}, we infer that proving that $\tr b \equiv 1$ amounts to show that there is a sequence $r_k \to 0^+$ such that
$
   \alpha_{r_k} \weaks 1    
$
weakly$^{\ast}$ in $L^{\infty} (]0, 1[ \times D, \Leb^3 \res ]0, 1 [ \times D)$ as $k \to + \infty$, for every $D$ open and bounded in~$\R^2$. Hence, we can conclude by choosing $r_k: = 2^{2-k}$.  
\subsubsection{Construction of a nontrivial solution of the initial-boundary value problem~\eqref{e:cex:ibvp2}}  
\label{sss:nt}
To exhibit a nontrivial solution of~\eqref{e:cex:ibvp2} we proceed as follows: first, we give the rigorous definition, next we make some heuristic remark and finally we show that $u$ is actually a distributional solution 
of~\eqref{e:cex:ibvp2}. 

We recall the decomposition $]0, 1[\times \Omega = \Lambda^+ \cup \Lambda^- \cup \mathcal S$, 
where $\Lambda^+$, $\Lambda^-$ and $\mathcal S$ are defined by~\eqref{e:lambdap},~\eqref{e:lambdam} and~\eqref{e:sigma}, respectively. We define the function $u: ]0, 1[ \times \Omega \to \R$ by setting 
\begin{equation}
\label{e:vu}
    u (t, r, y_1, y_2) : = 
    \left\{
    \begin{array}{ll}
    1 & (t, r, y_1, y_2) \in \Lambda^- \; \text{and} \; 
    b(t, r, y_1, y_2) \cdot (1, 0,0) =1 \\
    0 & \text{elsewhere in $]0, 1[ \times \Omega$}. 
     \end{array}
     \right.
\end{equation}
The heuristic idea behind this definition is as follows. We have defined the vector field $b$ in such a way that, although $b$ is overall outward pointing (namely, $\tr b > 0$), there are actually countably many regions where $b$ is 
 inward pointing (namely its $r$-component is strictly positive) which accumulate and mix at the domain boundary: these regions are represented by the dashed square in Figure~\ref{f:s1}.  The function $u$ is defined in such a way that $u$ is transported along the characteristics (which are well-defined for a.e. $(r, y)$ in the domain interior) and it is nonzero only on the regions where $b$ is inward pointing. As a result, although $b$ is overall outward pointing, it actually carries into the domain the nontrivial function $u$. This behavior is made possible by the breakdown of the $BV$ 
regularity of $b$ at the domain boundary. 

We now show that $u$ is a distributional solution of the initial-boundary value problem~\eqref{e:cex:ibvp2}. 
First, we observe that $u(t, \cdot) \weaks 0$ as $t \to 0^+$ and hence the weak continuity of $u$ with respect to the time implies that the initial datum is satisfied.  

We then set $C:= (u, bu)$ and we observe that 
$\divtx C =0$ on $\Lambda^+$ because $C$ is identically $0$ there. Next, we observe that the vector field $b$ is constant with respect to $t$ in $\Lambda^-$ and, by recalling that $u$ is defined as in~\eqref{e:vu}, we infer that $u$ is also constant with respect to $t$ in $\Lambda^-$. Hence, showing that $\divtx C\equiv 0$ in $\Lambda^-$ amounts to show that 
$
    \div (bu) \equiv 0
$
in $\Lambda^-$. This can be done by relying on the same arguments we used to obtain~\eqref{e:dfsc}. 

Finally, we observe that the normal vector to the surface $\mathcal S$ is (up to an arbitrary choice of the orientation)
$
  \vec n : = (1/ \sqrt{2}, -1/ \sqrt{2}, 0, 0). 
$
Hence, by construction the normal trace of $C$ is zero on both sides of the surface $\mathcal S$ and hence $\divtx C \res \mathcal S =0$.  
This concludes the proof of Theorem~\ref{t:cex2}. \hfill \qed

\subsection{Proof of Corollary~\ref{c:dpt}}
\label{ss:rm} 
We first describe the heuristic idea underlying the construction of the vector field $b$. Loosely speaking, we proceed as in the proof of Theorem~\ref{t:cex2}, but we modify the values of the ``building block" $\beta_k$ on the subinterval $r \in ]0, 2^{-k}[$. Indeed, instead of defining $\beta_k$ as in {\sc Step 1} of \S~\ref{sss:betak}, we introduce nontrivial components in the $(y_1, y_2)$-directions. These non-trivial components are reminiscent of the construction in Depauw~\cite{Depauw} and the resulting vector field can be actually regarded as a localized version of Depauw's vector field. In particular, they enable us to construct a solution that oscillates between $1$, $-1$ and $0$ and undergoes a finer and finer mixing as $r \to 0^+$. 

The technical argument is organized in two steps: in \S~\ref{sss:ldp} we introduce the ``localized version" of Depauw vector field, while in \S~\ref{sss:con} we conclude the proof of Corollary~\ref{c:dpt}. Before proceeding, we introduce the following notation:
\begin{itemize}
\item $Q_k$ is the square $(y_1, y_2) \in ]0, 2^{-k}[ \times ]0, 2^{-k}[$;
\item $S_k$ is the square $(y_1, y_2) \in ]0, 2^{2-k}[ \times ]0, 2^{2-k}[$.
\end{itemize}
\subsubsection{A localized version of Depauw~\cite{Depauw} vector field}
\label{sss:ldp}
We construct the vector field $\alpha_k$, which is defined on the cell ${(r, y_1, y_2) \in ]0, 2^{-k}[ \times Q_k}$. Also, for this construction we regard $r$ as a time-like variable and we describe how a given initial datum evolves under the action of $\alpha_k$.  The argument is divided into steps. 

{\sc Step 1:} we construct the ``building block" $a_k$, which is defined on the square $(y_1, y_2) \in ]- 2^{ -2-k}, 2^{-2-k}[ \times ]- 2^{-2-k}, 2^{-2-k}[$ by setting 
\begin{equation}
\label{e:ak}
       a_k (y_1, y_2) = 
       \left\{
       \begin{array}{ll}
       (0, - 2 y_1) & |y_1| > |y_2|  \\
       (2 y_2, 0)  & |y_1| < |y_2| \,. \\
       \end{array}
       \right.
\end{equation}
Note that $a_k$ takes values in $\R^2$, it is divergence free and it is tangent at the boundary of the square.

{\sc Step 2:} we define the function $\bar z_k: Q_k \to \R$ by considering the chessboard illustrated in Figure~\ref{f:dp4}, left part. The function $\bar z_k$ attains the value $-1$ and $1$ on white and black squares, respectively. 
\begin{figure}
\begin{center} 
\caption{The action of the vector field $\alpha_k (r, \cdot)$ on the solution $z_k$ on the interval $r \in ]0, 2^{-2-k}[$. The solution $z_k$ attains the values $1$ and $-1$ on black and white regions, respectively.} 
\psfrag{a}{\small{$\bar z_k $}}  \psfrag{b}{\small{$\alpha_k(r, \cdot)$}}
\psfrag{c}{\small{$z_k (2^{-2-k}, \cdot)$}}
\psfrag{d}{\small{$2^{-2-k}$}} 
\psfrag{x}{\small{$y_1$}} 
\psfrag{y}{\small{$y_2$}} 
\label{f:dp4}
\bigskip
\includegraphics[scale=0.4]{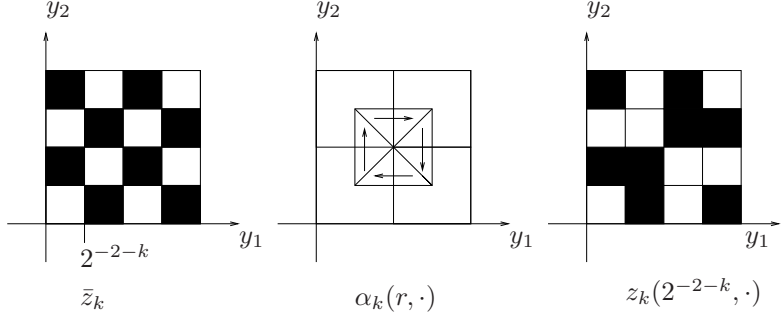}
\end{center}
\end{figure} 

{\sc Step 3:} we begin the construction of the vector field 
$
\alpha_k: ]0, 2^{-k}[  \times Q_k  \to \R^2.
$ 
If $r \in ]0, 2^{-2-k}[$, then $\alpha_k(r, \cdot)$ is defined by setting 
\begin{equation}
\label{e:alphak}
    \alpha_k (r, y_1, y_2) =
    \left\{
    \begin{array}{ll}
               a_k (y_1- 2^{-1-k}, y_2 -2^{-1-k}) \\
                \qquad \text{if $(y_1, y_2) \in ] 2^{-2-k}, 3 \cdot 2^{-2-k}   [ \times ] 2^{-2-k}, 3 \cdot 2^{-2-k}   [$}\\ \phantom{ciao} \\
               (0, 0)  \; \text{elsewhere on $Q_k$}. 
    \end{array}
    \right.
\end{equation}
See Figure~\ref{f:dp4}, central part, for a representation of the values attained by $\alpha_k$ on the interval $r \in ]0, 2^{-2-k}[$. 

We term $z_k$ the solution of the initial-boundary value problem
\begin{equation}
\label{e:zk1}
    \left\{
    \begin{array}{lll}
       \partial_r z_k + \divy ( \alpha_k z_k ) = 0 & 
        \text{in $]0, 2^{-k}[ \times Q_k$ } \\
       z_k = \bar z_k & \text{at $r=0$,} \\ 
    \end{array}
    \right.
\end{equation}
where $\bar z_k$ is defined as in {\sc Step 2}. Note that by construction $\divy \alpha_k =0$ and therefore the first line of~\eqref{e:zk1} is actually a transport equation. Hence, the value attained by the function $z_k$ can be determined by the classical method of characteristics. In particular, the function $z_k (2^{-2-k}, \cdot)$ is represented in Figure~\ref{f:dp4}, right part, and it attains the values $1$ and $-1$ on black and white squares, respectively. 

{\sc Step 4:} if $r \in ]2^{-2-k}, 3 \cdot 2^{-2-k}[$, then $\alpha_k(r, \cdot)$ is defined by setting 
$$
\alpha_k (r, y_1, y_2) =
               a_k (y_1- i 2^{-2-k}, y_2 - j 2^{-2-k}) 
$$               
if 
$$
    (y_1, y_2) \in ](i-1) 2^{-2-k}, (i+1)  \cdot 2^{-2-k}   [ \times ] (j-1) 2^{-2-k}, (j+1) \cdot 2^{-2-k}   [,
$$    
where $i, j$ can be either $1$ or $3$. See Figure~\ref{f:dp3}, central part, for a representation of the values attained by $\alpha_k$ on the interval $r \in  ]2^{-2-k}, 3 \cdot 2^{-2-k}[$. Note that by construction $\divy \alpha_k \equiv 0$ on $]0, 3 \cdot 2^{-2-k}[ \times Q_k$ and hence the solution $z_k$ of~\eqref{e:zk1} evaluated at $r = 3 \cdot 2^{-2-k}$ is as in Figure~\ref{f:dp3}, right part: as usual, the black and white squares represent the regions where $z_k (3 \cdot 2^{-2-k}, \cdot)$ attain the values $1$ and $-1$, respectively.
\begin{figure}
\begin{center} 
\caption{The action of the vector field $\alpha_k (r, \cdot)$ on the solution $z_k$ on the interval $r \in ]2^{-2-k}, 3 \cdot 2^{-2-k}[$. The solution $z_k$ attains the values $1$ and $-1$ on black and white regions, respectively.} 
\psfrag{a}{\small{$z_k(2^{-2-k}, \cdot) $}}  \psfrag{b}{\small{$\alpha_k(r, \cdot)$}}
\psfrag{c}{\small{$z_k (3 \cdot 2^{-2-k}, \cdot)$}}
\psfrag{x}{\small{$y_1$}} 
\psfrag{y}{\small{$y_2$}} 
\label{f:dp3}
\bigskip
\includegraphics[scale=0.4]{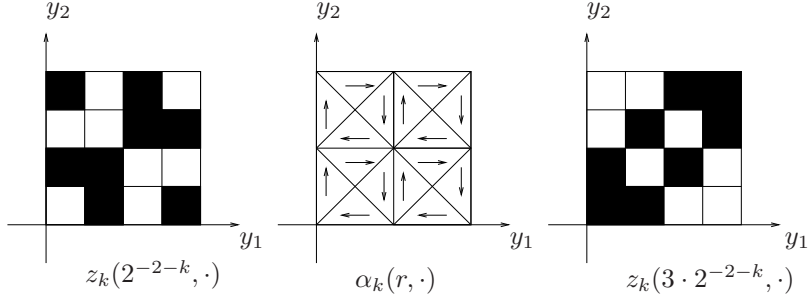}
\end{center}
\end{figure} 

{\sc Step 5:} if $r \in ]3 \cdot 2^{-2-k}, 4 \cdot 2^{-2-k}[$, then $\alpha_k(r, \cdot)$ is  again defined by~\eqref{e:alphak}. Hence, the values attained by $z_k(2^{-k}, \cdot)$ are those represented in Figure~\ref{f:dp1}, right part. 
\begin{figure}
\begin{center} 
\caption{The action of the vector field $\alpha_k (r, \cdot)$ on the solution $z_k$ on the interval $r \in ]3 \cdot 2^{-2-k}, 4 \cdot 2^{-2-k}[$. The solution $z_k$ attains the values $1$ and $-1$ on black and white regions, respectively.} 
\psfrag{a}{\small{$z_k(3 \cdot 2^{-2-k}, \cdot) $}}  \psfrag{b}{\small{$\alpha_k(r, \cdot)$}}
\psfrag{c}{\small{$z_k (4 \cdot 2^{-2-k}, \cdot)$}}
\psfrag{x}{\small{$y_1$}} 
\psfrag{y}{\small{$y_2$}} 
\label{f:dp1}
\bigskip
\includegraphics[scale=0.4]{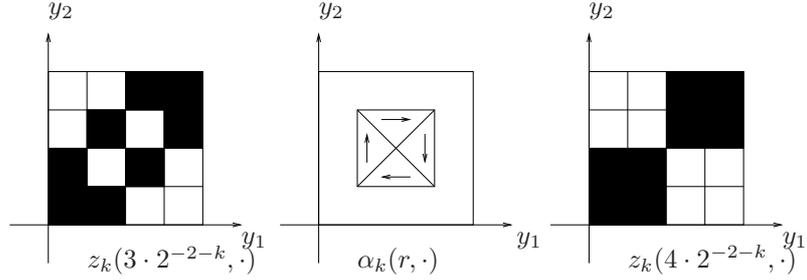}
\end{center}
\end{figure} 
\subsubsection{Conclusion of the proof}
\label{sss:con}
Loosely speaking, the proof of Corollary~\ref{c:dpt} is concluded by combining the construction described in \S~\ref{sss:ldp} with the proof of Theorem~\ref{t:cex2}. The argument is divided in four steps. 

{\sc Step A:} we define the vector field $\tilde \beta_k$ and the solution $u_k$ on $(r, y_1, y_2) \in ]0, 2^{-k}[ \times S_k$, where $S_k$ is the square $ ]0, 2^{2-k}[ \times  ]0, 2^{2-k}[$. 

We recall the definition of $B_k$ provided by~\eqref{e:dk} and we set
\begin{equation*}
\tilde \beta_k (r, y_1, y_2): =
\left\{
 \begin{array}{lll}
       \big(1, \alpha_k (r, y_1, y_2) \big) 
       & (y_1, y_2)  \in ]0, 2^{-k} [ \times ]0, 2^{-k}[ \\
       \phantom{\int} \\
       \big(1, \alpha_k (r, y_1-2 \cdot 2^{-k}, y_2) \big) 
       & (y_1, y_2)  \in ]2 \cdot 2^{-k}, 3 \cdot 2^{-k} [ \times ]0, 2^{-k}[ \\
       \phantom{\int} \\
        \big(1, \alpha_k (r, y_1, y_2- 2 \cdot 2^{-k}) \big) 
       & (y_1, y_2)  \in  ]0, 2^{-k}[ \times  ]2 \cdot 2^{-k}, 3 \cdot 2^{-k} [\\
       \phantom{\int} \\
        \big(1, \alpha_k (r, y_1- 2 \cdot 2^{-k}, y_2- 2 \cdot 2^{-k}) \big) 
       & (y_1, y_2)  \in   
       ]2^{-k}, 3 \cdot 2^{-k} [ \times  ]2^{-k}, 3 \cdot 2^{-k} [\\
       \phantom{\int} \\
       (-5, 0, 0) & (y_1, y_2) \in B_k \\
       \phantom{\int} \\
       (0, 0, 0) & \text{elsewhere in $S_k$} \\
 \end{array}
\right.
\end{equation*}
Note that, basically, the definition of $\tilde \beta_k$ is obtained from~\eqref{e:betak1} by changing the value of the vector field on $D_k$ and inserting as a component in the $(y_1, y_2)$-directions the vector field $\alpha_k$ constructed in \S~\ref{sss:ldp}. 

Also, we define the function $u_k$ by setting 
\begin{equation*}
 u_k (r, y_1, y_2): =
\left\{
 \begin{array}{lll}
       z_k (r, y_1, y_2) 
       & (y_1, y_2)  \in ]0, 2^{-k} [ \times ]0, 2^{-k}[ \\
       \phantom{\int} \\
       z_k (r, y_1- 2 \cdot 2^{-k}, y_2)  
       & (y_1, y_2)  \in ]2 \cdot 2^{-k}, 3 \cdot 2^{-k} [ \times ]0, 2^{-k}[ \\
       \phantom{\int} \\
       z_k (r, y_1, y_2- 2 \cdot 2^{-k}) 
       & (y_1, y_2)  \in  ]0, 2^{-k}[ \times  ]2 \cdot 2^{-k}, 3 \cdot 2^{-k} [\\
       \phantom{\int} \\
        z_k (r, y_1- 2 \cdot 2^{-k}, y_2- 2 \cdot 2^{-k}) 
       & (y_1, y_2)  \in   
       ]2 \cdot 2^{-k}, 3 \cdot 2^{-k} [ \times  
       ]2 \cdot 2^{-k}, 3 \cdot 2^{-k} [\\
       \phantom{\int} \\
       0  & \text{elsewhere in $S_k$}, \\
 \end{array}
\right.
\end{equation*}
where $z_k$ is the same function as in \S~\ref{sss:ldp}. 

{\sc Step B:} we define the vector field $\tilde \beta_k$ and the solution $u_k$ for $(r, y_1, y_2) \in ]2^{-k}, 2^{2-k}[ \times S_k$. 

We set 
$
   \tilde \beta_k (r, y_1, y_2) : = \beta_k  (r, y_1, y_2) 
$, where $\beta_k$ denotes the same vector field as in \S~\ref{sss:betak}. The function $u_k$ satisfies 
$$
    \partial_r u_k + \divy (\tilde \beta_k u_k) =0.
$$
Since $\divy \tilde \beta_k =0$, the values attained by $u_k$ for $(r, y_1, y_2) \in ]2^{-k}, 2^{2-k}[ \times S_k$ can be computed by the classical method of characteristics. To provide an heuristic intuition of the behavior of $u_k$, we refer to Figure~\ref{f:s1}, center and right part, and we point out that $u_k$ attains the value $0$ on white and black areas, while on dashed areas it attains the same values as in Figure~\ref{f:dp1}, right part. 
 
 {\sc Step C:} we extend $\tilde \beta_k$ and $u_k$ to $]0, 2^{2-k}[ \times \R^2$ by periodicity by proceeding as in~\eqref{e:per}.  
 
 {\sc Step D:} we finally define a vector field $b$ and the function $u$. We recall the decomposition~\eqref{e:dec} and we define $b$ as in~\eqref{e:brt} and~\eqref{e:btr}, replacing $\beta_k$ with $\tilde \beta_k$. Also, we define $u$ by setting 
 $$
 u(t, r, y_1, y_2) =
 \left\{
 \begin{array}{lll}
            u_k (r_, y_1, y_2) & \text{in $\Lambda^-$, when $r \in I_k$} \\
            0 & \text{in $\Lambda^+$}. 
 \end{array}
 \right.
 $$  
 By arguing as in the proof of Theorem~\ref{t:cex2}, one can show that $u$ and $b$ satisfy requirements i), $\dots$, v) in the statement of Theorem~\ref{t:cex2} and that moreover $\tr (bu) \equiv 0$. This concludes the proof of Corollary~\ref{c:dpt}. \hfill \qed
 
\section*{Acnkowledgments} 
The construction of the counter-examples exhibited in the proofs of Theorem~\ref{t:cex2} and Corollary~\ref{c:dpt} was inspired by a related example due to Stefano Bianchini. Also, the authors wish to express their gratitude to Wladimir Neves for pointing out reference~\cite{Boyer}. Part of this work was done when Spinolo was affiliated to the University of Zurich, which she thanks for the nice hospitality. Donadello and Spinolo thank the University of Basel for the kind hospitality during their visits. Crippa is partially supported by the
SNSF grant 140232, while Donadello acknowledges partial
support from the ANR grant CoToCoLa.
\bibliographystyle{plain}
\bibliography{boundary}

\begin{thebibliography}{10}

\bibitem{Ambrosio:trabv}
L.~Ambrosio.
\newblock Transport equation and {C}auchy problem for {$BV$} vector fields.
\newblock {\em Invent. Math.}, 158(2):227--260, 2004.

\bibitem{AmbrosioCrippa}
L.~Ambrosio and G.~Crippa.
\newblock Existence, uniqueness, stability and differentiability properties of
  the flow associated to weakly differentiable vector fields.
\newblock In {\em Transport equations and multi-{D} hyperbolic conservation
  laws}, volume~5 of {\em Lect. Notes Unione Mat. Ital.}, pages 3--57.
  Springer, Berlin, 2008.

\bibitem{AmbrosioCrippaManiglia}
L.~Ambrosio, G.~Crippa, and S.~Maniglia.
\newblock Traces and fine properties of a {$BD$} class of vector fields and
  applications.
\newblock {\em Ann. Fac. Sci. Toulouse Math. (6)}, 14(4):527--561, 2005.

\bibitem{AmbrosioFuscoPallara}
L.~Ambrosio, N.~Fusco, and D.~Pallara.
\newblock {\em Functions of bounded variation and free discontinuity problems}.
\newblock Oxford Mathematical Monographs. The Clarendon Press Oxford University
  Press, New York, 2000.

\bibitem{Anzellotti}
G.~Anzellotti.
\newblock Pairings between measures and bounded functions and compensated
  compactness.
\newblock {\em Ann. Mat. Pura Appl. (4)}, 135:293--318 (1984), 1983.

\bibitem{Bardos}
C.~Bardos.
\newblock Probl\`emes aux limites pour les \'equations aux d\'eriv\'ees
  partielles du premier ordre \`a coefficients r\'eels; th\'eor\`emes
  d'approximation; application \`a l'\'equation de transport.
\newblock {\em Ann. Sci. \'Ecole Norm. Sup. (4)}, 3:185--233, 1970.

\bibitem{Boyer}
F.~Boyer.
\newblock Trace theorems and spatial continuity properties for the solutions of
  the transport equation.
\newblock {\em Differential Integral Equations}, 18(8):891--934, 2005.

\bibitem{ChenFrid}
G.-Q. Chen and H.~Frid.
\newblock Divergence-measure fields and hyperbolic conservation laws.
\newblock {\em Arch. Ration. Mech. Anal.}, 147(2):89--118, 1999.

\bibitem{ChenTorresZiemer}
G.-Q. Chen, M.~Torres, and W.~P. Ziemer.
\newblock Gauss-{G}reen theorem for weakly differentiable vector fields, sets
  of finite perimeter, and balance laws.
\newblock {\em Comm. Pure Appl. Math.}, 62(2):242--304, 2009.

\bibitem{CDS:hyp}
G.~Crippa, C.~Donadello, and L.~V. Spinolo.
\newblock A note on the initial-boundary value problem for continuity equations
  with rough coefficients.
\newblock {Preprint. Available at http://cvgmt.sns.it/}.

\bibitem{Dafermos}
C.~M. Dafermos.
\newblock {\em Hyperbolic conservation laws in continuum physics}, volume 325
  of {\em Grundlehren der Mathematischen Wissenschaften [Fundamental Principles
  of Mathematical Sciences]}.
\newblock Springer-Verlag, Berlin, third edition, 2010.

\bibitem{CDL:notes}
C.~De~Lellis.
\newblock Notes on hyperbolic systems of conservation laws and transport
  equations.
\newblock In {\em Handbook of differential equations: evolutionary equations.
  {V}ol. {III}}, Handb. Differ. Equ., pages 277--382. Elsevier/North-Holland,
  Amsterdam, 2007.

\bibitem{CDL:bourbaki}
C.~De~Lellis.
\newblock Ordinary differential equations with rough coefficients and the
  renormalization theorem of {A}mbrosio [after {A}mbrosio, {D}i{P}erna,
  {L}ions].
\newblock {\em Ast\'erisque}, (317):Exp. No. 972, viii, 175--203, 2008.
\newblock S{\'e}minaire Bourbaki. Vol. 2006/2007.

\bibitem{Depauw}
N.~Depauw.
\newblock Non unicit\'e des solutions born\'ees pour un champ de vecteurs {BV}
  en dehors d'un hyperplan.
\newblock {\em C. R. Math. Acad. Sci. Paris}, 337(4):249--252, 2003.

\bibitem{diPernaLions}
R.~J. DiPerna and P.-L. Lions.
\newblock Ordinary differential equations, transport theory and {S}obolev
  spaces.
\newblock {\em Invent. Math.}, 98(3):511--547, 1989.

\bibitem{GiraultRidgway}
V.~Girault and L.~Ridgway Scott.
\newblock On a time-dependent transport equation in a {L}ipschitz domain.
\newblock {\em SIAM J. Math. Anal.}, 42(4):1721--1731, 2010.

\bibitem{Leoni}
G.~Leoni.
\newblock {\em A first course in {S}obolev spaces}, volume 105 of {\em Graduate
  Studies in Mathematics}.
\newblock American Mathematical Society, Providence, RI, 2009.

\bibitem{Mischler}
S.~Mischler.
\newblock On the trace problem for solutions of the {V}lasov equation.
\newblock {\em Comm. Partial Differential Equations}, 25(7-8):1415--1443, 2000.

\end{thebibliography}

\end{document}